\documentclass[letterpaper,11pt]{amsart}
\usepackage{latexsym}
\usepackage{amsmath,amssymb,amsthm}
\usepackage{amscd}
\usepackage[sorted]{amsrefs}

\title{Half-turn linked pairs of isometries of hyperbolic $4$-space}
\author{Andrew E. Silverio}
\address{Department of Mathematical Sciences, University of Cincinnati, 45221}
\email{andrew.silverio@uc.edu}

\newtheorem{theorem}{Theorem}[section]
\newtheorem{corollary}[theorem]{Corollary}

\newtheorem{definition}[theorem]{Definition}
\newtheorem{lemma}[theorem]{Lemma}
\numberwithin{equation}{section}

\begin{document}
\begin{abstract}
 In this paper we develop a complete theory of factorization for isometries of hyperbolic 4-space. Of special interest is the case where a pair of isometries is linked, that is, when a pair of isometries can be expressed each as compositions of two involutions, one of which is common to both isometries.
\par  Here we develop a new theory of hyperbolic pencils and twisting planes involving a new geometric construction, their half-turn banks. This enables us to give complete results about each of the pair-types of isometries and their simultaneous factorization by half-turns. That is, we provide geometric conditions for each such pair to be linked by half-turns. The main result gives a necessary and sufficient condition for any given pair of isometries to be linked. We also provide a procedure for constructing a half-turn linked pair of isometries of $\mathbb{H}^4$ that do not restrict to lower dimensions, yielding an example that gives a negative answer to a question raised by Ara Basmajian and Karan Puri.
\end{abstract}


\maketitle
\section{Introduction}
A pair of isometries $A,B$ of the hyperbolic space $\mathbb{H}^n$ is said to be linked if there are involutions $\alpha$, $\beta$ and $\gamma$ such that $A=\alpha\beta$ and $B=\beta\gamma$. If furthermore $\alpha$, $\beta$ and $\gamma$ have $(n-2)$-dimensional fixed point set, then the pair is said to be linked by  half-turns. In dimensions $2$ and $3$, every pair of isometries is linked. Factoring by half-turns is used to determine the discreteness of the group $\langle A,B\rangle$ using the Gilman-Maskit algorithm in dimension $2$, the non-separating-disjoint-circles condition and the compact-core-geodesic-intersection condition in dimension $3$ \cites{algorithm,criteria,discon}. Unfortunately in dimension $4$, not all pairs are linked. Hence we develop a complete theory of half-turn factorization of orientation preserving isometries of $\mathbb{H}^4$.
\par To each plane $P$ in $\mathbb{H}^4$ corresponds a unique orientation preserving involution $H_P$ called the half-turn about $P$ whose fixed point set is $P$. A space of planes can be associated to each isometry that is constructed from a new concept of hyperbolic pencils. This space is called the half-turn bank of an isometry, and it locates all the planes in which half-turn factoring is possible.
\begin{theorem} \label{theorem:onepointone}
A pair of orientation preserving isometries of $\mathbb{H}^4$ is linked by half-turns if and only if they have a common element in their half-turn banks.
\end{theorem}
Note that the half-turn bank of an isometry is a set of circles defined by its fixed points and invariants, not by its factorization.  It is the introduction of the concept of a half-turn bank and the analysis of the half-turn bank for each type of isometry and the fact that the half-turn banks are constructible that make the theorem useful.
\par We also provide geometric criteria for a plane to be an element of the half-turn bank of an isometry. These criteria allow geometric conditions for various combinations of pairs of isometries to be linked. They are listed in section \ref{section:conditions}.
\par In section \ref{subsection:lorentzpencils}, Theorem \ref{theorem:disprovearasorry} is proven to show that a pair of isometries can be linked even if the group it generates does not restrict to lower dimensions, giving a negative answer to the question raised by Basmajian/Puri. A specific example is constructed.

\section{Organization}
\par This paper is organized as follows. In Section \ref{section:typesofisometry}, the classification of types of isometries in $\mathbb{H}^4$ are listed and defined in terms of reflection and other factoring. The half-turn is also defined. Some types are collected into so called \textit{atomic} types to consolidate a few definitions, while are others are broken further into subtypes. In Section \ref{section:pencils}, the new definitions of permuted, invariant and twisting (hyperbolic) pencils are defined for each type and subtype of isometry. The properties of hyperbolic pencils relevant to the main results are also listed. In Section \ref{section:halfturnbanks}, the half-turn bank of each type of isometry is introduced and defined. Theorem \ref{theorem:revisedkaran} is proved yielding the fact that the order of the linked pairs is negligible. Geometric criteria for a plane to be in a half-turn bank are listed and proven in Section \ref{subsection:geometriccriteria}. In Section \ref{section:justify}, planes are combined to show the result of composing two half-turns. The main theorem is restated and proven in Corollary \ref{corollary:punchline}.
\par In Section \ref{section:conditions}, some conditions for a pair to be linked are listed. The pairs of isometries are arranged according to the type or class.
\par In Section \ref{section:disproveara}, the hyperboloid model of hyperbolic space is used to find invariant subspaces of an isometry which relate to its hyperbolic pencils. The invariant spaces help in constructing a linked pair that does not restrict to lower dimensions, implying Theorem \ref{theorem:disprovearasorry}.

\section{Types of isometries in $\mathbb{H}^4$} \label{section:typesofisometry}

In general dimensions, the isometries of $\mathbb{H}^n$ classify into elliptic, hyperbolic and parabolic. Any isometry can be expressed as a composition of at most $n+1$ reflections. In dimension $4$, the classification of nontrivial isometries can subdivided into six types: type-I elliptic, type-II elliptic; pure hyperbolic; pure loxodromic; pure parabolic; and screw parabolic (See \cites{karan,ratcliffe}).
\par Pure parabolic and pure hyperbolic isometries are translations generalized from dimension $2$. A type-I elliptic isometry has fixed-point set that form a plane. Together, type-I elliptic, pure parabolic and pure hyperbolic are called \emph{atomic} isometries since they can be expressed as compositions of two reflections. A screw parabolic or pure loxodromic isometry can be decomposed uniquely into two commuting atomic isometries called its translation part and rotational part. The fixed-point set of its rotational part is called its twisting plane. A type-II elliptic isometry $\rho$ is conjugate to a matrix with two diagonal blocks of $2\times 2$ matrices in $\mathrm{SO}(2)$. Hence there is an unordered pair of angles $\alpha$ and $\beta$ associated with $\rho$. If $\alpha=\beta=\pi$, then $\rho$ is conjugate to the antipodal map and $\rho$ is an involution. Otherwise if both $\alpha$ and $\beta$ are not integer multiples of $\pi$, then there is a unique unordered pair of planes left invariant by isometry $\rho$.
\par Using the upper half-space model and restricting the action to the boundary, a parabolic isometry $\tau$ is conjugate to M\"obius transformation of the form $x\mapsto Ax+b$ where $b\neq 0$, $A \in \mathrm{SO}(3)$ and $Ab=b$ (see \cites{ratcliffe}). Assume for the moment that $\tau(x)=Ax+b$. Then $b$ is called the \emph{direction of $\tau$}. The Euclidean line connecting $0$ to $b$ in $\widehat{\mathbb{R}^3}$ is left invariant by $\tau$ and so is the plane $P_\tau$ in $\mathbb{H}^4$ bounded by this Euclidean line. Note that if $A$ is nontrivial, then $\tau$ is screw parabolic and $P_\tau$ is its twisting plane. The map $x\mapsto Ax$ is the rotational part of $\tau$ while $x\mapsto x +b$ is the translation part.
\par The following are definitions of the six types of isometries of $\mathbb{H}^4$ and half-turn, in terms of the number of reflection factors or atomic isometry factors (See \cites{karan,mythesis}). The advantage of these definitions is to locate the half-turn and reflection factors.

\begin{definition}
\par A \emph{type-I elliptic isometry of $\mathbb{H}^4$} is a composition of two reflections across distinct hyperplanes that intersect in a plane.

\par Let $\rho$ be a type-I elliptic isometry of $\mathbb{H}^4$. The \emph{twisting plane of $\rho$} is the set of fixed points of $\rho$ in $\mathbb{H}^4$.

\par Let $P$ be a plane in $\mathbb{H}^4$. The \emph{half-turn $H_P$ about $P$} is the elliptic isometry $H_P$ that is a composition of reflections across an orthogonal pair of hyperplanes that intersect in $P$.

\par An elliptic isometry $\rho$ of $\mathbb{H}^4$ is called \emph{of type-II} if $\rho$ fixes a unique point in $\mathbb{H}^4$.

\par A pure \emph{hyperbolic isometry} $\delta$ of $\mathbb{H}^4$ is a composition of exactly two reflections across ultra-parallel hyperplanes.

\par A \emph{pure loxodromic isometry $\gamma$} is a hyperbolic isometry of $\mathbb{H}^4$ that can be expressed as $\delta\rho$ where $\delta$ is a pure hyperbolic isometry leaving the axis of $\gamma$ invariant and $\rho$ is a type-I elliptic isometry whose twisting plane contains the axis of $\gamma$.

\par A \emph{pure parabolic isometry} $\tau$ of $\mathbb{H}^4$ is a composition of exactly two reflections across hyperplanes that are tangent to a unique point at infinity.

\par A \emph{screw parabolic isometry $\gamma$} is a parabolic isometry of $\mathbb{H}^4$ that can be expressed as $\tau\rho$ where $\tau$ is a pure parabolic isometry fixing the fixed point of $\gamma$, and $\rho$ is a type-I elliptic isometry whose twisting plane is left invariant by $\tau$.
\end{definition}

\par The $\rho$ factor in the definitions of either pure loxodromic or screw parabolic isometries is called its \emph{rotational part} while the other factor ($\delta$ or $\tau$) is called its \emph{translation part}. These factors commute and are unique. The fixed-point set of the rotational part of $\gamma$ is left invariant by $\gamma$ and is called the \emph{twisting plane} of $\gamma$.
\par An orientation preserving type-II elliptic isometry can be further classified into either an involution (antipodal map) or having a unique pair of invariant planes that intersect orthogonally to a single point. The latter subtype has also a pair of associated angles to the invariant planes; one such angle is not an integer multiple of $\pi$.

\section{Hyperbolic Pencils} \label{section:pencils}

\par The concept of pencils can be found from complex analysis, inversive geometry and projective geometry. Pencils are sometimes called coaxial family/systems of circles. In this section, new constructions and type of pencils are introduced to aid the process of factoring isometries into half-turns. The permuted, invariant, dual and twisting pencils of an isometry for each type are defined in this section. All types except type-II elliptic have permuted and dual pencils. The atomic isometries have invariant pencils while those with rotational parts have twisting pencils.

\begin{definition}
Let $\rho$ be a type-I elliptic isometry of $\mathbb{H}^4$ with a twisting plane $P$. The (elliptic) \emph{permuted pencil of $\rho$} is the set
\[\mathcal{F}_\rho = \left \{ \partial h \subset \partial \mathbb{H}^4 : h \text{ is a hyperplane containing } P \right\}.\]
The (elliptic) \emph{invariant pencil of $\rho$} is the set
\[\mathcal{T}_\rho = \left \{ \partial h \subset \partial \mathbb{H}^4 : h \text{ is a hyperplane orthogonal to } P \right\}.\]

The \emph{dual pencil of $\rho$} is the set
\[ \begin{split}
\mathcal{D}_\rho & = \left \{ \partial Q \subset \partial \mathbb{H}^4 : Q \text{ is a plane orthogonally intersecting} \right. \\
\nonumber & \quad \left. P \text{ in a unique point} \right\}.
\end{split}\]

\end{definition}

\begin{definition}
Let $\gamma$ be a hyperbolic isometry of $\mathbb{H}^4$ with axis $L$. The \emph{permuted pencil $\mathcal{F}_\gamma$ of $\gamma$} is the set
\[ \begin{split}
\mathcal{F}_\gamma &= \left \{ \partial h_x \subset \partial \mathbb{H}^4 :\, x \in L \text{ and } h_x \text{ is the orthogonal} \right. \\
	& \quad \left. \text{ complement of } L \text { through } x\right\}.
\end{split}\]
The \emph{invariant pencil of $\gamma $} is the set
\[\mathcal{T}_{\gamma}= \left \{ \partial t \subset \partial \mathbb{H}^4 : t \text{ is a hyperplane containing } L\right\}.\]
The \emph{dual pencil of $\gamma $} or \emph{pencil dual to $\mathcal{F}_\gamma$} is the set
\[\mathcal{D}_\gamma =\left\{ \partial P \subset \partial \mathbb{H}^4 : P \text{ is a plane containing } L \right\}.\]
\end{definition}

\par It is possible to define the permuted, dual, and invariant pencils of any parabolic isometry that fixes any point other than $\infty$ (see \cites{mythesis}). For easy reading, we define the pencils by conjugating its fixed point to $\infty$.
\par Let $\tau$ be the map sending $x$ to $Ax+b$ in $\widehat{\mathbb{R}^3}=\partial \mathbb{H}^4$ where $A\in \mathrm{SO}(3)$, $Ab=b$ and $b\neq 0$.The Euclidean line connecting $0$ to $b$ is left invariant by $\tau$. Moreover if $A$ is trivial, then any other Euclidean line parallel to $\mathrm{span}\{b\}$ is also left invariant by $\tau$, and so are the Euclidean planes parallel to $b$.

\begin{definition}
Let $\tau$ be the map of $\widehat{\mathbb{R}^3}$ that sends $x$ to $Ax+b$ in $\widehat{\mathbb{R}^3}$ where $A\in \mathrm{SO}(3)$, $Ab=b$ and $b\neq 0$. The \emph{permuted pencil of $\tau$} is the set
\[\mathcal{F}_\tau = \{(\text{Euclidean planes orthogonal to $b$})\cup \{\infty\} \}.\]
The \emph{invariant pencil of $\tau$} is the set
\[\mathcal{T}_\tau = \{(\text{Euclidean planes parallel to $b$})\cup \{\infty\} \}\]
The \emph{dual pencil of $\tau$} is the set
\[\mathcal{D}_\tau = \{(\text{Euclidean lines parallel to $b$})\cup \{\infty\} \}\]
\end{definition}

\begin{definition}
Let $\gamma$ be either pure loxodromic or screw parabolic isometry of $\mathbb{H}^4$ with rotational part $\rho$. The \emph{twisting pencil of $\gamma$} is defined to be the permuted pencil of $\rho$. The twisting pencil of $\gamma$ is denoted $ \mathcal{R}_\gamma $.
\end{definition}

\subsection{Visualizing pencils}
While all the pencils can be defined without assuming special fixed points, it may be difficult to imagine. Hence, we give examples of how they can possibly look. The formulas here are restrictions to the boundary model $\widehat{\mathbb{R}^3}$. By Poincar\'e extension, they extend uniquely to isometries of $\mathbb{H}^4$.

\par Let $\delta$ be the map $x\mapsto \lambda x$ where $\lambda >0$. Then $\delta$ fixes $0$ and $\infty$. Its pencils are as follows.
\begin{align*}
\mathcal{F}_\delta &= \{\text{Euclidean spheres centered at the origin}\}\\
\mathcal{T}_\delta &= \{(\text{Euclidean planes through the origin}) \cup \{\infty\}\}\\
\mathcal{D}_\delta &= \{(\text{Euclidean lines through the origin}) \cup \{\infty\}\}
\end{align*}
\par Let $\rho$ be the elliptic isometry that fixes the $z$-axis of $\mathbb{R}^3 \subset \widehat{\mathbb{R}^3}$ point\-wise and rotating by angle $\pi/4$. Its pencils are as follows.
\begin{align*}
\mathcal{F}_\rho &= \{(\text{Euclidean planes containing the $z$-axis})\cup \{\infty\}\}\\
\mathcal{T}_\rho &= \{(\text{Euclidean planes orthogonal to the $z$-axis})\cup \{\infty\}\} \cup\\
\nonumber & \quad \ \{\text{spheres centered at the $z$-axis}\}\\
\mathcal{D}_\rho &= \{\text{circles centered at the $z$-axis but contained in a plane in $\mathcal{T}_\rho$}\}
\end{align*}
The composition $\delta\rho=\rho\delta$ has $\rho$ as its rotational part and $\delta$ as its translation part. If $b=(0,0,1)$, then $ x \mapsto \rho(x)+b$ has $\rho$ as its rotational part while $x \mapsto x +b$ is its translation part.

\subsection{Properties of pencils} \label{subsection:propertiesofpencils}

For an atomic type of isometry, the permuted pencil serves as the source for reflection factoring. The rest of the nontrivial isometries need four factors of reflections. If $h$ is a sphere at infinity or a hyperplane inside $\mathbb{H}^4$, denote the reflection across $h$ with $R_h$. The following theorem is well-known in many books about hyperbolic geometry such as \cites{fenchel,maskit,karan,ratcliffe}.
\begin{theorem}
Let $\tau$ be an atomic isometry of $\mathbb{H}^4$. Then for each $h \in \mathcal{F}_\tau$ there exist $h_1,h_2 \in \mathcal{F}_\tau$ such that $\displaystyle \tau =R_{h_1}R_{h}=R_{h}R_{h_2}$.
\end{theorem}

\par The elements of permuted pencils fill up the boundary while the hyperplanes that they bound fill up the hyperbolic $4$-space. If there is a point $x$ outside the axis, twisting plane or fixed points of $\gamma$, there is a unique element of $\mathcal{F}_\gamma$ that contains $x$.


\par The rest of the properties are listed as follows.

\par Let $\rho$ be an atomic isometry.
\begin{enumerate}
	\item If $t\in \mathcal{T}_\rho $, then $\rho(t)=t$ and $t$ is orthogonal to each $s\in \mathcal{F}_\rho$.
	\item If $t,u \in \mathcal{T}_\rho$ such that $u \cap t$ has more than one point but $u \neq t$, then $u \cap t \in \mathcal{D}_\rho$.
	\item If a sphere $s\subset \widehat{\mathbb{R}^3}$ contains some $c\in \mathcal{D}_\rho$, then $s\in \mathcal{T}_\rho$.
	\item For each $d\in \mathcal{D}_\rho$, there are $t,u \in \mathcal{T}_\rho$ such that $d = t\cap u$.
	\item For each pair $p,q \in \mathcal{D}_\rho$ with $p\neq q$, there is a unique $t\in \mathcal{T}_\rho$ containing $p\cup q$.
\end{enumerate}

Let $\gamma$ be either a pure loxodromic or a screw parabolic isometry of $\mathbb{H}^4$ with twisting plane $P$.
\begin{enumerate}
	\item If $t\in \mathcal{R}_\gamma $, then $\gamma(t)\in \mathcal{R}_\gamma$, $\gamma(t)\cap t = \partial P$ and $t$ is orthogonal to each $s\in \mathcal{F}_\gamma$.
	\item If $t,u \in \mathcal{R}_\gamma$ with $u \neq t$, then $u \cap t = \partial P$.
	\item If a sphere $s\subset \widehat{\mathbb{R}^3}$ contains $\partial P$, then $s\in \mathcal{R}_\gamma$.
	\item For each $d \in \mathcal{D}_\gamma\setminus \{ \partial P\}$, there is a unique $t\in \mathcal{R}_\gamma$ containing $d\cup (\partial P)$.
	\item If $\rho$ is the rotational part of $\gamma$, then $\mathcal{F}_\rho \subset \mathcal{T}_\gamma$.
\end{enumerate}

\section{Half-turn banks} \label{section:halfturnbanks}
\par In order to find half-turns or planes in which involution factoring is possible, a space of planes is collected similarly to the elements of the pencils. The space $\mathcal{K}_\gamma$, called the half-turn bank of $\gamma$, is a collection of circles derived from the pencils, but not prevalent enough to be called another pencil. The goal is to show that all half-turn factorizations of $\gamma$ come from $\mathcal{K}_\gamma$.

\par Whenever there is no confusion, we refer to the elements of $\mathcal{K}_\gamma$ as either planes in $\mathbb{H}^4$ or circles in $\partial \mathbb{H}^4$.

\begin{definition}
Let $\rho$ be an atomic isometry of $\mathbb{H}^4$. The \emph{half-turn bank of $\rho$} is the set
\[\mathcal{K}_\rho=\left\{ s\cap t \subset \partial \mathbb{H}^4 : s \in \mathcal{F}_\rho \text{ and } t \in \mathcal{T}_\rho \right\}.\]
The \emph{half-turn bank of the identity map of $\mathbb{H}^4$} is the set $\mathcal{K}_{\mathrm{id}}$ consisting of the circles that bound all planes in $\mathbb{H}^4$.
\end{definition}

\begin{definition}
Let $\gamma$ be either pure loxodromic or screw parabolic isometry of $\mathbb{H}^4$. The \emph{half-turn bank of $\gamma$} is the set \[ \mathcal{K}_\gamma =
 \left \{ s\cap t \subset \widehat{\mathbb{R}^3} : s \in \mathcal{F}_\gamma \text{ and } t \in \mathcal{R}_\gamma \right\}.\]
\end{definition}

\par As previously mentioned, an orientation preserving elliptic isometry $\rho$ is conjugate to an element of $\mathrm{SO}(4)$ that has two $2\times 2$ blocks of matrices in $\mathrm{SO}(2)$. The associated angles $\alpha$ and $\beta$ tell whether $\rho$ is type-I or type-II. The condition for type-II elliptic is that $\alpha$ and $\beta$ are nonzero. If exactly one of them is neither $\pi$ nor $0$, then $\rho$ can be factored uniquely into two commuting type-I elliptic isometries $\rho_1$ and $\rho_2$ whose twisting planes are orthogonally intersecting in a single point. Thus $\rho$ also leaves these two twisting planes invariant. If $\alpha=\beta=\pi$, then $\rho$ is conjugate to the antipodal map of the $3$-sphere and is also an involution.
\begin{definition} \label{definition:factorizationoftype2elliptic}
Let $\rho= \rho_1 \rho_2$ be an orientation preserving type-II elliptic isometry of $\mathbb{H}^4$ with an associated angle not an integer multiple of $\pi$ so that $\rho_1$ and $\rho_2$ are the unique type-I elliptic elements whose twisting planes intersect orthogonally in the fixed point. The \emph{half-turn bank of $\rho$} is the set
\[\mathcal{K}_\rho=\left\{ s\cap t \subset \partial \mathbb{H}^4 : s \in \mathcal{F}_{\rho_1} \text{ and } t \in \mathcal{F}_{\rho_2} \right\}.\]
\end{definition}

\begin{definition}
Let $\rho$ be an isometry of $\mathbb{H}^4$ that is conjugate to the antipodal map with fixed point $x$. The \emph{half-turn bank of $\rho$} is the set
\[\mathcal{K}_\rho=\left\{\partial P \subset \partial \mathbb{H}^4 : P \text{ is a plane containing } x\right\}.\]
\end{definition}

If $k$ is the boundary at infinity of a plane $P$, let $H_k$ be the half-turn about $P$ as well. The following shows that the half-turn bank is a source for half-turn factoring of an orientation preserving isometry.

\begin{theorem} \label{theorem:revisedkaran} 
Let $\gamma$ be an orientation preserving isometry of $\mathbb{H}^4$. Then for each $k\in \mathcal{K}_\gamma$, there exist $k_1,k_2\in \mathcal{K}_\gamma$ such that $\gamma = H_{k_1}H_{k}=H_{k}H_{k_2}$.
\end{theorem}
\begin{proof}
\par If $\gamma$ has a rotational part, let $\tau$ be the translation part of $\gamma$ and $\rho$ be the rotational part of $\gamma$.  If $ k \in\mathcal{K}_\gamma $, there are $s\in \mathcal{F}_\gamma$ and $t\in \mathcal{R}_\gamma$ such that $k=s\cap t$. Then there are $s_1,s_2 \in\mathcal{F}_\tau = \mathcal{F}_\gamma$ such that $\tau =R_{s_1}R_{s}=R_{s}R_{s_2}$ and there are $t_1,t_2 \in \mathcal{F}_\rho =\mathcal{R}_\gamma$ such that $\rho =R_{t_1}R_{t}=R_{t}R_{t_2}$. Each sphere in $\{s, s_1, s_2\}$ is orthogonal to each sphere in $\{t, t_1, t_2\}$. The following relations hold.
\begin{align*}
R_{s}R_{t_{1}}&= R_{t_{1}}R_{s} & H_k &= R_{s}R_{t}\\
R_{t_{2}}R_{s}&= R_{s}R_{t_2} & H_k &= R_{t}R_{s}
\end{align*}
 Let $k_1=s_1\cap t_1$ and $k_2=s_2\cap t_2$. Then $H_{k_1}= R_{s_1}R_{t_1}=R_{t_1}R_{s_1}$ and $H_{k_2}=R_{s_2}R_{t_2}=R_{t_2}R_{s_2}$. Since $\gamma=\tau\rho=\rho\tau$,
\begin{align*}
\gamma &= \tau \rho & \gamma&=\rho \gamma\\
\nonumber & = R_{s_1}R_{s}R_{t_1}R_{t} & \nonumber & = R_{t}R_{t_2}R_{s}R_{s_2} \\
\nonumber & = R_{s_1}R_{t_1}R_{s}R_{t} & \nonumber & = R_{t}R_{s}R_{t_2}R_{s_2} \\
\nonumber & = H_{k_1}H_{k} & \nonumber & = H_{k}H_{k_2}
\end{align*}

\par If $\gamma$ is atomic and $k\in \mathcal{K}_\gamma$, there are $s\in \mathcal{F}_\gamma$ and $t\in \mathcal{T}_\gamma$ such that $k=s\cap t$. Then there are $s_1,s_2 \in \mathcal{F}_\gamma$ such that $\gamma =R_{s_1}R_{s}=R_{s}R_{s_2}$. Let $k_1=s_1\cap t$ and $k_2=s_2\cap t$. Since $s$, $s_1$ and $s_2$ are orthogonal to $t$, $H_{k_1}= R_{s_1}R_{t}=R_{t}R_{s_1}$ and $H_{k_2}=R_{s_2}R_{t}=R_{t}R_{s_2}$. Then,
\begin{align*}
\delta &= R_{s_1}R_{s} & \delta &= R_{s}R_{s_2}\\
\nonumber &= R_{s_1}R_{t} R_{t}R_{s} & \nonumber &= R_{s}R_{t} R_{t}R_{s_2}\\
\nonumber &= H_{k_1}H_{k} & \nonumber &= H_{k}H_{k_2}.
\end{align*}
\par If $\gamma$ is a type-II elliptic involution, let $P$ be the plane
bounded by $k$. Then there is a unique plane $Q$ orthogonal to $P$
at the fixed point of $\gamma$. Let $k_1=\partial Q$. Then
$\gamma= H_k H_{k_1}=H_{k_1}H_k$ as shown on Corollary
\ref{corollary:type2involution}. Let $k_2=k_1$ so that
$\gamma=H_kH_{k_2}$.
\par If $\gamma$ is type-II elliptic but not an involution, let
$\gamma=\rho_1\rho_2$ be its factorization in Definition \ref{definition:factorizationoftype2elliptic}.
Then there are $s\in \mathcal{F}_{\rho_1}$ and $t\in \mathcal{F}_{\rho_2}$
such that $k = s \cap t$. There are also $s_1,s_2\in \mathcal{F}_{\rho_1}$ such that
$\rho_1=R_{s_1}R_s=R_sR_{s_2}$ and $t_1,t_2\in \mathcal{F}_{\rho_2}$ such that
$\rho_2=R_{t_1}R_t=R_tR_{t_2}$. Since $s,s_1,s_2 \in \mathcal{T}_{\rho_2}$
and $t,t_1,t_2 \in \mathcal{T}_{\rho_1}$, each sphere in $\{s,s_1,s_2\}$
is orthogonal to each sphere in $\{t,t_1,t_2\}$.
Let $k_1=s_1\cap t_1$ and $k_2=s_2\cap t_2$. Similar to the case where
$\gamma$ has a rotational part, we have $\gamma=H_{k_1}H_k=H_kH_{k_2}$.
\end{proof}

\subsection{Alternative geometric definitions} \label{subsection:geometriccriteria}
The elements of half-turn banks are defined as intersections of spheres. In lower dimensions, involution-\-factoring is constructed from the lines perpendicular to the axis of an isometry \cites{algorithm,maskit}. The following statements are the analogous geometric definitions.

\begin{theorem} \label{theorem:secretiffelliptic}
Let $\rho$ be a type-I elliptic isometry of $\mathbb{H}^4$ with twisting plane $P$. Then $Q$ is a plane orthogonal to $P$ through a line if and only if $\partial Q \in \mathcal{K}_\rho$.
\end{theorem}
\begin{proof}
If $Q$ intersects $P$ in a line $\ell$, $P\cup Q$ forms a unique hyperplane $h$. There is a hyperplane $t$ orthogonal to $h$ through $Q$. Then $P$ is orthogonal to $t$ through $\ell$ since $P\subset h$. It follows that $\partial t \in \mathcal{T}_\rho$ and $\partial h \in \mathcal{F}_\rho$. Since $Q=t\cap h$, $\partial Q \in \mathcal{K}\rho$.
\par Conversely if $\partial Q \in \mathcal{K}_\rho$, let $h \in \mathcal{F}_\rho$ and $t \in \mathcal{T}_\rho$ such that $\partial Q = h\cap t$. Let $\hat h$ and $\hat t$ be the hyperplanes bounded by $h$ and $t$, respectively. The hyperplane $\hat t$ and the plane $P$ are orthogonal through a line. Then $(\partial P) \cap t$ is a set of two points contained in $\partial Q$ since $\partial P \subset h$. Both $P$ and $Q$ are planes in the hyperplane bounded by $h$; they intersect in the line $P\cap \hat t$. Since $\hat t$ is orthogonal to $P$ and $Q \subset \hat h$, then $P$ and $Q$ are orthogonal through the line $P\cap \hat t$.
\end{proof}

\begin{theorem} \label{theorem:secretiffpurehyperbolic}
Let $\delta$ be a pure hyperbolic isometry of $\mathbb{H}^4$ with axis $L$. Then $P$ is a plane orthogonal to $L$ if and only if $\partial P\in \mathcal{K}_\delta$.
\end{theorem}
\begin{proof}
Suppose $P$ is a plane orthogonal to $L$. Let $b$ be the intersection point of $P$ and $L$. There are orthogonal lines $v_P,w_P \subset P$ through $b$. Let $y$ be the unique line orthogonal to $v_P$, $w_P$ and $L$ through $b$. There is a unique hyperplane $h_b$ orthogonal to $L$ through $b$. Then $y \subset h_b$. Let $\hat t$ be the unique hyperplane spanned by $L$, $w_P$ and $v_P$. Then $h_b\cap \hat t =P$ since $w_P, v_P \subset h_b\cap \hat t$. Let $t=\partial \hat t$ and $s=\partial h_b$. Then $t\cap s =\partial P$. Since $h_b$ is orthogonal to $L$, $s$ is in $\mathcal{F}_\delta$. Similarly, the fact that $\hat t$ contains $L$ implies that $t\in \mathcal{T}_\delta$.
\par Conversely if $\partial P \in \mathcal{K}_\delta$, there are $h \in \mathcal{F}_\delta$ and $t \in \mathcal{T}_\delta$ such that $\partial P = h\cap t$. Let $\hat h$ and $\hat t$ be the hyperplanes bounded by $h$ and $t$ respectively. Then $\hat h$ is orthogonal to $L$ through a point $x$. Since $\hat t$ contains $L$ and $P$, $x\in L \cap P$. As a subset of $\hat h$ containing $x$, $P$ is orthogonal to $L$.
\end{proof}

\begin{theorem} \label{theorem:secretiffpureparabolic}
Let $\tau$ be a pure parabolic isometry of $\mathbb{H}^4$ with fixed point $v\in \widehat{\mathbb{R}^3}$. Then $c$ is a circle in $\widehat{\mathbb{R}^3}$ such that $v\in c \subset h$ for some $h \in \mathcal{F}_\tau$ if and only if $c \in \mathcal{K}_\tau$.
\end{theorem}
\begin{proof}
Let $c$ be a circle in $\widehat{\mathbb{R}^3}$ such that $v\in c \subset h$ for some $h \in \mathcal{F}_\tau$. It suffices to show that $c\in \mathcal{K}_\tau$ in the case where $v=\infty$. Then $h$ is a Euclidean plane and $c$ is a Euclidean line. Let $t$ be the set $\bigl \{x+ \lambda \tau(0): x \in c \setminus \{\infty \}, \, \lambda \in \mathbb{R} \bigr\} \cup \{\infty \}$. Then $t$ is orthogonal to $h$ through $c$ so $t \in \mathcal{T}_\tau$ and $c=h \cap t$. Hence, $c\in \mathcal{K}_\tau$.
\par Conversely if $c \in \mathcal{K}_\tau$, there are $h \in \mathcal{F}_\tau$ and $t \in \mathcal{T}_\tau$ such that $c = h\cap t$. Then $c\subset h$. Both $h$ and $t$ contain $v$ so $v\in c $.
\end{proof}

\begin{theorem} \label{theorem:secretiffpureloxodromic}
Let $\gamma$ be a pure loxodromic isometry of $\mathbb{H}^4$ with axis $L$ and twisting plane $P$. Then $Q \subset \mathbb{H}^4$ is a plane orthogonal to both $P$ and $L$ through a line in $P$ if and only if $\partial Q \in \mathcal{K}_{\gamma}$.
\end{theorem}
\begin{proof}
Suppose $Q$ is a plane orthogonal to both $P$ and $L$ such that $Q\cap P$ is a line. Let $\ell = Q \cap P$. If $Q$ is orthogonal to $L$, then $\ell$ and $L$ are perpendicular lines in $P$. Let $x$ be the intersection point $Q$ and $L$. Let $\ell '$ be the unique line in $Q$ perpendicular to $\ell$ through $x$. Likewise $\ell '$ is perpendicular to $L$. There is a unique line $\ell_4$ perpendicular to all $L$, $\ell$ and $\ell '$. In particular, $\ell_4$ is the unique line orthogonal to the hyperplane spanned by $P$ and $Q$ passing through $x$. Let $h$ be the hyperplane spanned by $\ell$, $\ell '$ and $\ell_4$. Since $L$ is perpendicular to all $\ell$, $\ell '$ and $\ell_4$, $h$ is the orthogonal complement of $L$ through $x$. Then $\partial h \in \mathcal{F}_\gamma$. Let $t$ be the hyperplane spanned by $L$, $\ell$ and $\ell '$. Therefore $P\subset t$ and $t\cap h = Q$. Hence $\partial t \in \mathcal{R}_\gamma$ and $\partial Q \in \mathcal{K}_\gamma$.
\par Conversely if $\partial Q \in \mathcal{K}_\gamma$, there are $h \in \mathcal{F}_\gamma$ and $t \in \mathcal{R}_\gamma$ such that $\partial Q = h\cap t$. Let $\hat h$ and $\hat t$ be the hyperplanes bounded by $h$ and $t$ respectively. Then $\hat h$ is orthogonal to to $P$ through a line since $\partial P$ is orthogonal to every element of $\mathcal{F}_\gamma$. Whereas $\hat t$ contains $P$ so $Q=\hat h \cap \hat t$ must also intersect $P$ in the line $P\cap \hat h$ since $Q\cap P=\bigl(\hat h \cap \hat t \bigr)\cap P=\hat h \cap \bigl(\hat t \cap P \bigr)=\hat h \cap P$. Then $P$ is orthogonal to $\hat h$ through $P\cap Q$, so $Q\subset \hat h$ implies that $P$ is also orthogonal to $Q$ through the line $P\cap Q$.
\par Since $L$ is a subset of $\hat t$ and $L$ is orthogonal to $\hat h$ through $x$, $Q$ intersects $L$ at $x$. But $Q \subset \hat h$ so $Q$ is orthogonal to $L$.
\end{proof}

\begin{theorem} \label{theorem:secretiffscrewparabolic}
Let $\gamma$ be a screw parabolic isometry of $\mathbb{H}^4$ with fixed point $v$ and twisting plane $P$. Then $Q \subset \mathbb{H}^4$ is a plane orthogonal to $P$ through a line bounded by $v$ if and only if $\partial Q \in \mathcal{K}_{\gamma}$.
\end{theorem}
\begin{proof}
Suppose $Q$ is a plane orthogonal to $P$ through a line bounded by $v$. There is a hyperplane $t$ contaning $Q\cup P$ since they intersect in a line. Let $\ell = Q \cap P$. There is a point $x\in \partial \ell \setminus \{v\}$. Then there is a unique $h_x \in \mathcal{F}_\gamma$ that contain $x$. Since $Q$ is orthogonal to $P$, $Q$ is contained in the hyperplane bounded by $h_x$. But $Q,P \subset t$ so $\partial t \in \mathcal{R}_\gamma$ and $\partial Q = (\partial t)\cap h_x$. Hence $\partial Q \in \mathcal{K}_\gamma$.
\par Conversely if $\partial Q \in \mathcal{K}_\gamma$, there are $h \in \mathcal{F}_\gamma$ and $t \in \mathcal{R}_\gamma$ such that $\partial Q = h\cap t$. Let $\hat h$ and $\hat t$ be the hyperplane bounded by $h$ and $t$ respectively. Then $\hat h$ is orthogonal to $P$ through a line bounded by $v$. Also $\hat t$ contains $P$, so $Q=\hat h \cap \hat t$ implies $Q\cap P =\bigl(\hat h \cap t\bigr)\cap P=\hat h \cap \bigl(\hat t \cap P \bigr)=\hat h \cap P$. Then $P$ is orthogonal to $\hat h$ through $P\cap Q$ and hence to $Q$. Since $v\in h \cap t$, then $v\in \partial Q$; $v\in \partial P$ implies that $v$ bounds the line $P\cap Q$.
\end{proof}

\section{Properties of half-turn banks} \label{section:justify}
The half-turn bank of an isometry of hyperbolic space serves as a collection of circles in which the isometry can be factored into a product of two half-turns. In this section, we show that any half-turn factorization of an orientation preserving isometry comes from the defined set called the half-turn bank.

\par In section \ref{subsection:commonperp2}, the common perpendicular line across a pair of ultra-parallel $(n-2)$-dimensional planes is constructed. In section \ref{subsection:combinations} all possible combinations of pairs of planes in $\mathbb{H}^4$ are enumerated to show the result of composing half-turns about those planes. Together with classification of isometries, the combinations tell how the involved planes intersect whenever an isometry is factored into a product of two half-turns. Section \ref{subsection:justify} has the details of proving that all half-turn factorizations come from half-turn banks. The theorem implies that a pair of isometries is linked by a half-turn if and only if they have a common circle in their half-turn banks.

\subsection{Codimension-$2$ version of common perpendicular} \label{subsection:commonperp2}
If $P$ and $Q$ are ultra-parallel hyperplanes in $\mathbb{H}^4$, there is a unique line perpendicular to them (See \cites{fenchel, algorithm, maskit}). If $P$ and $Q$ are ultra-parallel lines in $\mathbb{H}^4$, there is also a unique line perpendicular to them. The same result is shown for a pair of ultra-parallel $(n-2)$-dimensional subplanes of $\mathbb{H}^n$.

\begin{theorem} \label{theorem:commonperptwo}
For each pair of ultra-parallel planes $\alpha, \beta \in \mathbb{H}^4$, there is a unique line orthogonal to both $\alpha$ and $\beta$.
\end{theorem}
 The definitions of half-turn banks can be extended to orientation preserving isometries of $\mathbb{H}^n$, but to justify its name, a more general statement is stated as follows.
\begin{theorem} \label{theorem:commonperpgeneral}
Let $n\geq 2$ be an integer. For each pair of ultra-parallel $(n-2)$-dimensional planes $\alpha, \beta \in \mathbb{H}^n$, there is a unique line orthogonal to both $\alpha$ and $\beta$.
\end{theorem}
\begin{proof}
Theorem \ref{theorem:commonperpgeneral} implies Theorem \ref{theorem:commonperptwo}, so constructing the common orthogonal line in a general setting is enough for a case in the proof of Theorem \ref{theorem:exhaustive}. Let $P$ and $Q$ be ultra-parallel $(n-2)$-dimensional planes of $\mathbb{H}^n$. If $n=2$, then $P$ and $Q$ are distinct points, and the unique line connecting them are vacuously orthogonal to $P$ and $Q$. If $n=3$, then $P$ and $Q$ are ultra-parallel lines so there is a unique line perpendicular to both of them (see \cites{fenchel,algorithm,maskit}). If $n\geq 4$, we use the hyperboloid model of $\mathbb{H}^n$ embedded in the Lorentzian space $\mathbb{R}^{n,1}$. Then $P$ and $Q$ extend to $(n -1)$-dimensional vector subspaces $P'$ and $Q'$ respectively. Since $P$ and $Q$ are ultra-parallel, the span of $P'\cup Q'$ is at least of dimension $n$. However, its dimension can not add up to $2n-2$ but only to a maximum of $n+1$.

Therefore, the dimension of the space-like subspace $P'\cap Q'$ is either $n-2$ or $n-3$. If it is $n-2$, then $\mathrm{span}(P' \cup Q')$ is $n$-dimensional and has $P'$ and $Q'$ as its hyperplanes with a unique common perpendicular line (See \cites{ratcliffe}). If $\mathrm{dim}(P'\cap Q')$ is $n-3$, then $N=\bigl(P'\cap Q'\bigr)^L$, the Lorentz orthogonal complement of $P'\cap Q'$, is a $4$-dimensional time-like vector subspace nontrivially intersecting both $P'$ and $Q'$ due to their large dimensions. Since $N$ contains ${P'}^L$ and ${Q'}^L$, which respectively intersect $P'$ and $Q'$ trivially, $N\cap P'$ and $N\cap Q'$ are $2$-dimensional subspaces. If they are time-like, then there is a unique line in $\mathbb{H}^n$ perpendicular to both $\mathbb{H}^n \cap N\cap P'$ and $\mathbb{H}^n \cap N\cap Q'$. We must show that $N\cap P'$ and $N\cap Q'$ are time-like vector subspaces.

\par Pick a time-like vector vector $x\in N$. Then $x \notin {P'}^L \cup {Q'}^L$ since ${P'}^L$ and ${Q'}^L$ are space-like. Let $w\in {P'}^L$ and $v \in {Q'}^L$ be nontrivial elements so they are linearly independent with $x$. Since ${P'}^L$ and ${Q'}^L$ are subsets of $N$, the spans of $\{x,w\}$ and $\{x,v\}$ are subspaces of $N$. Then the vector $-\frac{\langle x,w \rangle_L}{\langle w,w \rangle_L}w+x$ is an element of both $N$ and $P'$ since it is Lorentz orthogonal to $w\in{P'}^L$, and is a linear combination of $w$ and $x$. Then its Lorentz norm can be computed as follows.

\begin{align*}
\left \lVert -\frac{\langle x,w \rangle_L}{\langle w,w \rangle_L}w+x \right \rVert^2_L &= \frac{\langle x,w \rangle ^2_L}{\lVert w \rVert ^2 _L} -2\frac{\langle x,w \rangle ^2_L}{\lVert w \rVert ^2 _L} +\lVert x \rVert ^2 _L \\
\nonumber &=-\frac{\langle x,w \rangle ^2_L}{\lVert w \rVert ^2 _L} + \lVert x \rVert ^2 _L
\end{align*}
Since $x$ is time-like and $w$ is space-like, the quantity above is negative, making the vector time-like. Hence, $N\cap P'$ is a $2$-dimensional time-like subspace of $P'$. By replacing the vector $w$ with $v$, the same arguments show that $N\cap Q'$ is also a $2$-dimensional time-like subspace of $Q'$.

\end{proof}

\subsection{Combinations of a pair of planes} \label{subsection:combinations}
 There are a few ways for two planes in $\mathbb{H}^4$ to intersect. They can intersect in a line, a point, or in a point at infinity. The last case is not technically an intersection but points in the two planes can be arbitrarily near.

\par Let $P$, $Q$ be distinct planes in $\mathbb{H}^4$. There are four ways they can intersect or not intersect: ultra-parallel; both tangent at a unique point at infinity; in a line; or in a single point.
Each of these combinations forms a unique type of isometry of $H_PH_Q$. The following are the results.

\begin{lemma}
Let $P$ and $Q$ be ultra-parallel planes in $\mathbb{H}^4$. Then $H_PH_Q$ is a hyperbolic isometry.
\end{lemma}
\begin{proof}
If $P$ and $Q$ are ultra-parallel, then there is a unique line $N$ in $\mathbb{H}^4$ that is orthogonal to both $P$ and $Q$. Both $H_P$ and $H_Q$ leave $N$ invariant so $H_PH_Q$ must also leave $N$ invariant. That leaves two options for $H_PH_Q$: either hyperbolic or elliptic. Suppose $H_PH_Q$ has a fixed point $y$ inside $\mathbb{H}^4$. Then $H_P(y)=H_Q(y)$, so the midpoint between $y$ and $H_P(y)$ is an element of both $P$ and $Q$. This contradicts the hypothesis that $P$ and $Q$ are ultra-parallel.
\end{proof}
\begin{lemma}
Let $P$ and $Q$ be tangent planes in $\mathbb{H}^4$. Then $H_PH_Q$ is a parabolic isometry.
\end{lemma}
\begin{proof}
Let $P$ and $Q$ be distinct planes in $\mathbb{H}^4$ so that their boundaries meet at one point $x$ at infinity. By conjugation, we may assume $x=\infty$ of $\widehat{\mathbb{R}^3}$. Then $P'$ and $Q'$ are Euclidean lines. If $P'$ and $Q'$ form a Euclidean plane, there are several Euclidean lines perpendicular to both $P'$ and $Q'$ but all of them identify a unique direction or vector $v \in \mathbb{R}^3$. If they do not form a plane, there is a unique Euclidean line $N$ orthogonal to $P'$ and $Q'$, which identifies a direction $v \in \mathbb{R}^3$. In either case, all Euclidean lines with same direction as $v$ are left invariant by $H_PH_Q$. Furthermore, $H_PH_Q$ is a parabolic translation $(x\mapsto x +2v, \quad x \in \mathbb{R}^3 )$ if $P'$ and $Q'$ are coplanar or a screw parabolic isometry with twisting plane bounded by $N$ if not. So if $P$ and $Q$ are tangent, $H_PH_Q$ is definitely parabolic.
\end{proof}
\begin{lemma}
Let $P$ and $Q$ be distinct planes in $\mathbb{H}^4$ intersecting in a line. Then $H_PH_Q$ is a type-I elliptic isometry whose twisting plane is orthogonal to the hyperplane containing $P\cup Q$ through $P\cap Q$.
\end{lemma}
\begin{proof}
Let $P$ and $Q$ be planes of $\mathbb{H}^4$ intersecting in a line $L$. Then they form a unique hyperplane $h$ whose boundary is an element of both permuted pencils of $H_P$ and $H_Q$. It follows that $h$ is left invariant, and $L$ is fixed pointwise by $H_PH_Q$. The plane $\tau$ orthogonal to $h$ through $L$ is also left by $H_PH_Q$, but since $L$ is fixed pointwise, $H_PH_Q$ also fixes $\tau$ pointwise. To illustrate it with the model $\widehat{\mathbb{R}^3}$ as the boundary at infinity, assume $P'$ is a Euclidean line and $Q'$ is a circle intersecting $P'$ in two points $y_1$ and $y_2$ . Then $Q'\cup P'$ forms a unique Euclidean plane $\hat h$. There is a unique circle $\hat \tau$ passing through $y_1$ and $y_2$ centered at their midpoint, and inside the Euclidean plane orthogonal to $\hat h$ through $P'$. Then $\hat \tau$ is orthogonal to both $P'$ and $Q'$ through $P'\cap Q'$. Both half-turns $H_P$ and $H_Q$ flip $\hat \tau$ across $P'\cap Q'$ so the composition $H_PH_Q$ fixes $\hat \tau$ pointwise. Since $\hat \tau$ is the unique plane orthogonal to $h$ through $P\cap Q$, $\hat \tau$ is the only fixed point set of $H_PH_Q$. It follows that $H_PH_Q$ is a type-I elliptic isometry. Moreover, its twisting plane is orthogonal to $\hat h$.
\end{proof}

\begin{lemma} \label{lemma:singlepointtype2}
Let $P$ and $Q$ be distinct planes in $\mathbb{H}^4$ intersecting in a single point. Then $H_PH_Q$ is a type-II orientation preserving elliptic isometry.
\end{lemma}
\begin{proof}
Suppose $P$ and $Q$ are planes $\mathbb{H}^4$ intersecting in a unique point $p$. Then for each $x\in \mathbb{H}^4\setminus\{p\}$, we show that $H_PH_Q(x)\neq x$. Suppose first that $x\in Q$. Then $H_PH_Q(x)=H_P(x)$ which is not equal to $x$ since $x\notin P$. For the rest of this combination of $P$ and $Q$, assume that $x\in \mathbb{H}^4\setminus Q$. If it happens that $H_Q(x) \in P$, then $H_PH_Q(x)=H_Q(x)\neq x$. If $H_Q(x) \notin P$ but $x\in P$, then $H_PH_Q(x)\notin P$ so $H_PH_Q(x)\neq x$.
\par The last possibility is when both $H_Q(x)$ and $x$ are outside $Q\cup P$. We prove that $H_P$ can not map $H_Q(x)$ back to $x$. Suppose $H_PH_Q(x)=x$. Let $m$ be the midpoint between $H_Q(x)$ and $x$, and let $L$ be the line connecting $x$ to $H_Q(x)$. Then $m$ is in $Q$ since it is the midpoint between $x$ and $H_Q(x)$. Likewise $m$ must also be in $P$ as $H_P(x)=H_Q(x)$. There is only one point in $P\cap Q$ so this midpoint must be $p$. Furthermore, $L$ is orthogonal to both $P$ and $Q$ through $p$. It follows that the hyperplane orthogonal to $L$ through $p$ contains both $P$ and $Q$. Since intersecting planes within a hyperplane must meet in at least a line, $P\cap Q$ must have at least a line which has more points other than $p$. This contradicts that $P$ and $Q$ intersect only in a point.
\end{proof}

\par Since $H_PH_Q$ is type II elliptic, one might wonder where the invariant planes or lines are located. If $P$ and $Q$ are orthogonal complements of each other, then $H_Q$ and $H_P$ respectively rotate them half-way around. Otherwise, $H_PH_Q$ has a canonical pair of invariant planes. To locate these invariant planes, we can use the ball model of $\mathbb{H}^4$ embedded in $\mathbb{R}^4$ which is conformal to both the Euclidean geometry and the spherical geometry of $S^3$. The problem simplifies further if $P$ and $Q$ are conjugated to intersect at the origin. The following lemma locates the canonical invariant planes relative to $P$ and $Q$.

\begin{lemma} \label{lemma:type2planes}
Let $P$ and $Q$ be $2$-dimensional vector subspaces of $\mathbb{R}^4$ intersecting trivially. Then there are $2$-dimensional vector subspaces $\tau_1$ and $\tau_2$ orthogonal complements of each other such that $\tau_1$ and $\tau_2$ are orthogonal to $P$ and $Q$ through two separate lines.
\end{lemma}
\begin{proof}
Let $U_P= S^3 \cap P $ and $U_Q= S^3\cap Q $ be unit (Euclidean) circles. The Euclidean inner product restricted to $U_P\times U_Q$ has a maximum value realized by a pair $(v_P, v_Q)$ since $U_P\times U_Q$ is compact and the inner product is continuous. The vectors $v_P,v_Q$ can be augmented by $w_P\in P$ and $w_Q\in Q$ so that the sets $\{v_P , w_P\}$ and $\{v_Q , w_Q\}$ are orthonormal bases. Let $\tau_1 = \mathrm{span}\{v_P , v_Q\}$ and $\tau_2 = \mathrm{span}\{w_P , w_Q\}$. The dimension of $\tau_i$ is $2$ since $v_P \neq v_Q$. It follows that $\tau_1 \cap \tau_2= \{ 0 \}$ and they intersect $P$ and $Q$ through two separate lines. What is left to show is that $\tau_1$ and $\tau_2$ are orthogonal to $P$, $Q$ and each other. It is sufficient to show that $\langle v_P, w_Q \rangle = \langle v_Q, w_P \rangle =0$.
\par The function $\theta \mapsto \langle (\cos \theta) v_P +(\sin \theta) w_P, v_Q \rangle$ is continuous and smooth with a maximum at $\theta = 0$. Its derivative $\theta \mapsto \langle -(\sin \theta) v_P +(\cos \theta) w_P, v_Q\rangle$ therefore has a zero value at $\theta =0$. Hence $\langle v_Q, w_P \rangle = 0$. Similarly the function $\theta \mapsto \langle (\cos \theta) v_Q + (\sin \theta) w_Q, v_P \rangle$ has derivative $\theta \mapsto \langle- (\sin \theta) v_Q +(\cos \theta) w_Q, v_P \rangle$ with a zero value at $\theta =0$, implying that $\langle w_Q, v_P \rangle = 0$.
\end{proof}

\begin{corollary} \label{corollary:type2involution}
Let $P$ and $Q$ be orthogonal planes in $\mathbb{H}^4$ intersecting in a unique point. Then $H_PH_Q$ is a type-II elliptic involution that leaves every line through $P\cap Q$ invariant.
\end{corollary}
\begin{proof}
Since $P$ and $Q$ are orthogonal, they are left invariant by both $H_P$ and $H_Q$. The half-turns are involutions themselves so applying $H_PH_Q$ twice to $P\cup Q$ is the identity map on $P\cup Q$. The conformal ball model can be used to conjugate $H_PH_Q$ into an element of $\mathrm{SO}(4)$, with $p$ corresponding to the origin. Then $P$ and $Q$ form vector spaces that are orthogonal complements of each other. Any Euclidean orthonormal bases of them can be combined into an orthonormal basis $\mathcal{B}$ of $\mathbb{R}^4$. If $x\in \mathbb{H}^4=B_4$ is outside $P\cup Q$, it can be expressed as a linear combination of vectors in $\mathcal{B}\subset \partial (P\cup Q)$. The composition $\left(H_PH_Q\right)^2$ as an element of $\mathrm{SO}(4)$ therefore maps $x$ back to itself. Hence $H_PH_Q$ is an involution.

\par A matrix of
$\mathrm{SO}(n)$ has its inverse and transpose equal, but if it is also an involution, then
$H_PH_Q$ as a matrix is also symmetric and thus diagonalizable (Spectral Theorem). There are four linearly independent eigenvectors that correspond to lines in
$\mathbb{H}^4$ that are pairwise perpendicular through
$p$. The diagonal entries are all $-1$ since a $1$ value would make $H_PH_Q$ have more than one fixed point and other values would make the matrix not in $\mathrm{SO}(4)$. Each vector is hence mapped into its opposite. In the conformal ball model of hyperbolic space, the action of $H_PH_Q$ on lines passing through $p$ is reflection across $p$.
\end{proof}

\begin{corollary} \label{corollary:type2notinvolution}
Let $P$ and $Q$ be non-orthogonal planes in $\mathbb{H}^4$ intersecting in a unique point. Then $H_PH_Q$ is a type-II elliptic isometry with a unique unordered pair of invariant planes orthogonal to each other through the fixed point of $H_PH_Q$.
\end{corollary}
\begin{proof}
If $P$ and $Q$ intersect only in one point, then $H_PH_Q$ is type II elliptic isometry. There are planes $\tau_1$ and $\tau_2$ that are orthogonal to each other through the fixed point of $H_PH_Q$ and also orthogonal to both $P$ and $Q$ through separate lines. Both $H_P$ and $H_Q$ leave $\tau_1$ and $\tau_2$ invariant since they are orthogonal to the half-turns' fixed point sets. Then $\tau_1$ and $\tau_2$ are also left invariant by $H_PH_Q$.

\par It must be shown that $\tau_1$ and $\tau_2$ are the unique invariant planes. In order to show their uniqueness, $H_PH_Q$ can be conjugated to a $4\times 4$ matrix in $\mathrm{SO}(4)$ so that the upper-left and lower-right $2\times 2$ blocks are elements of $\mathrm{SO}(2)$. Since $H_PH_Q$ has only one fixed point, either both these blocks are diagonal matrices with $-1$ in their entries or one of these blocks is non-diagonalizable.
\par Recall that the construction of $\tau_1$ allows it to have $v_P\in P$ and $v_Q\in Q$ so that the angle between $v_P$ and $v_Q$ is at minimum. If $P$ and $Q$ are not orthogonal, this angle is less than $\pi/2$. The action of $H_PH_Q$ on $\tau_1$ is a composition of reflections across $\mathrm{span}\{ v_P\}$ and $\mathrm{span}\{ v_Q\}$. Thus, one of the non-diagonalizable blocks corresponds to the rotation of $\tau_1$ in an angle other than $0$ and $\pi$. Then $\tau_1$ is the unique invariant plane of $H_PH_Q$ with minimum angle or rotation. The orthogonal complement of $\tau_1$ is $\tau_2$ which is also unique.
\end{proof}

\subsection{The half-turn bank is exhaustive} \label{subsection:justify}
\par In this section, we show that every half-turn factorization of an orientation preserving isometry comes from its half-turn bank. The proof still uses cases but the lemmas from section \ref{subsection:combinations} restrict the possibilities for how a pair of planes intersect.
\begin{theorem} \label{theorem:exhaustive}
Let $\gamma$ be an orientation preserving isometry of $\mathbb{H}^4$. Then for every half-turn factorization $H_PH_Q$ of $\gamma$, the circles $\partial P$ and $\partial Q$ are elements of $\mathcal{K}_\gamma$.
\end{theorem}
\begin{proof}
\par The proofs depend on the class of isometry of $\gamma$. \begin{itemize}
\item $\gamma$ is hyperbolic.
\par If $\gamma$ is hyperbolic, let $L$ be its axis. The only combination for $P$ and $Q$ is that they are ultra-parallel. The common perpendicular line of $P$ and $Q$ is left invariant by $H_PH_Q$ so it must be $L$. Then both $P$ and $Q$ are orthogonal to the axis of $\gamma$. Let $h_P$ be the hyperplane spanned by $L$ and $P$. Similarly, let $h_Q$ be the hyperplane spanned by $L$ and $Q$. Then $\partial h_P\in \mathcal{F}_{H_P}$ and $\partial h_Q\in \mathcal{F}_{H_Q}$ with $P\subset h_P $; $Q\subset h_Q $.
\par If $h_P=h_Q$, there are $f_P \in \mathcal{F}_{H_P} $ and $f_Q \in \mathcal{F}_{H_Q} $ such that $H_P=R_{f_P}R_{h_P} $ and $H_Q=R_{h_Q}R_{f_Q} $ are reflection factorizations. If $h_P=h_Q$, then $\gamma = H_PH_Q=R_{f_P}R_{f_Q}$ is a pure hyperbolic isometry with $\partial f_P, \partial f_Q \in\mathcal{F}_\gamma $ and $\partial h_P= \partial h_Q \in \mathcal{T}_\gamma$. So $\partial P= \partial (h_P\cap f_P) \in \mathcal{K}_\gamma $ and $\partial Q=\partial (h_Q\cap h_Q) \in \mathcal{K}_\gamma $.
\par If $h_P\cap h_Q$ is a plane $\tau$, then $\tau$ intersects $P$ and $Q$ in two different lines. The half-turns $H_P$ and $H_Q$ reflect $\tau$ across these lines and thus leave $\tau$ invariant. If $\gamma$ is pure loxodromic, then $h_P\neq h_Q$, so the twisting plane of $\gamma$ matches $\tau$. Still, the hyperplanes $h_Q^\perp$ orthogonal to $h_Q$ through $Q$ and $h_P^\perp$ orthogonal to $h_P$ through $P$ are also orthogonal to $L$ so $\partial h_P^\perp, \partial h_Q^\perp \in \mathcal{F}_\gamma$. Both $h_P$ and $h_Q$ contain $\tau$ so $\partial h_P, \partial h_Q \in \mathcal{R}_\gamma$. Since $P=h_P \cap h_P^ \perp $ and $Q= h_Q\cap h_Q^\perp $, then $\partial P, \partial Q \in \mathcal{K}_\gamma$.
\item $\gamma$ is parabolic.
\par If $\gamma$ is parabolic, then $P$ and $Q$ are tangent at infinity. For simpler illustration, assume that the fixed point of $\gamma$ is $\infty$. Then the boundaries $(\partial P, \partial Q)$ of $P$ and $Q$ are straight non-crossing Euclidean lines in $\mathbb{R}^3$. Any Euclidean line commonly perpendicular to the boundaries of $P$ and $Q$ forms the same direction exactly equal to that of $\gamma$. Recall that the direction of the map $x \mapsto Ax + b$ is the Euclidean line spanned by $b$ in $\mathbb{R}^3$. Let $B$ the direction of $\gamma$. Then $B$ is either Euclidean-parallel or equal to any common perpendicular between $\partial P$ and $\partial Q$. Let $h_P$ and $h_Q$ be the Euclidean planes orthogonal to $B$ through $\partial P$ and $\partial Q$ respectively. Let $f_P$ be the Euclidean plane orthogonal to $h_P$ through $\partial P$ and let $f_Q$ be Euclidean plane orthogonal to $h_Q$ through $\partial Q$. Then $h_P, f_P \in \mathcal{F}_{H_P}$ and $h_Q,f_Q \in \mathcal{F}_{H_Q}$. Since $(h_P,f_P)$ and $(h_Q,f_Q)$ are pairwise orthogonal, $H_P=R_{h_P}R_{f_P}$ and $H_Q=R_{f_Q}R_{h_Q}$ are reflection factorizations.
\par If $f_P=f_Q$, then $\gamma=H_PH_Q=R_{h_P}R_{h_Q}$ is pure parabolic, so $f_P=f_Q \in \mathcal{T}_\gamma$ and $h_P, h_Q \in \mathcal{F}_\gamma$. If $f_P\cap f_Q $ is a Euclidean line $\tau$, it is left invariant by $H_P$ and $H_Q$ and hence by $\gamma$. This implies that $\tau$ bounds the twisting plane of $\gamma$ and therefore $f_P, f_Q \in \mathcal{R}_\gamma$. Still, $\partial P, \partial Q \in \mathcal{K}_\gamma$.
\item $\gamma$ is type-I elliptic.
\par If $\gamma$ is type-I elliptic, then $P$ and $Q$ intersect in a line. Let $\tau$ be the twisting plane of $\gamma$. There is also a unique hyperplane $h$ spanned by $P$ and $Q$. Then $\tau$ is orthogonal to $h$ through $P\cap Q$. The planes $P$, $Q$ and $\tau$ pairwise intersect at $P\cap Q$ while $\tau$ is orthogonal to both $P$ and $Q$ through $P\cap Q$. Let $h_P$ be the hyperplane spanned by $\tau$ and $P$; let $h_Q$ be the hyperplane spanned by $Q$ and $\tau$. It follows that $\partial h\in \mathcal{T}_\gamma$ and $\partial h_P, \partial h_Q \in \mathcal{F}_{\gamma}$. Since $P=h_P \cap h$ and $Q=h_Q \cap h$, we have $\partial P, \partial Q \in \mathcal{K}_\gamma$.
\item $\gamma$ is type-II elliptic.
\par If $\gamma$ is type-II elliptic, then $P$ and $Q$ intersect in a unique point $p$. Using the conformal ball model of $\mathbb{H}^4$ inside $\mathbb{R}^4$, we may assume that $p$ is the origin. Then $P$ and $Q$ extend to Euclidean planes that intersect only at the origin. There are planes $\tau_1$ and $\tau_2$ that are orthogonal complements of each other and also orthogonal to both $P$ and $Q$ through separate lines (Lemma \ref{lemma:type2planes}). Let $h_P$, $ h_Q $, $h_1$ and $h_2$ be hyperplanes defined as follows.
\begin{align*}
 h_{P} &= \mathrm{span}(\tau_1\cup P) & h_{Q} &= \mathrm{span}(\tau_2\cup Q)\\
 h_1 &= \mathrm{span}(Q\cup \tau_1) & h_2 &= \mathrm{span}(P\cup \tau_2)
\end{align*}
The configuration of planes and hyperplanes yields to $\tau_1 \subset h_P \cap h_1$, $\tau_2 \subset h_Q \cap h_2$, $P=h_P \cap h_2$ and $Q=h_Q \cap h_1$.

\par There are two options for $\gamma$ and the planes $\tau_1,\tau_2$. Either $\gamma$ has an associated angle that is not an integer multiple of $\pi$ (Corollary \ref{corollary:type2notinvolution}) or $\gamma$ is an involution (Corollary \ref{corollary:type2involution}). In the former case, $\gamma$ has a unique pair of invariant planes that must match $\tau_1$ and $\tau_2$. Then $\gamma=\rho_1\rho_2$ where $\rho_1$ and $\rho_2$ are type-I elliptic isometries whose respective twisting planes are $\tau_1$ and $\tau_2$. Since $\partial h_2, \partial h_P \in \mathcal{F}_{\rho_2}$ and
$\partial h_1, \partial h_Q \in \mathcal{F}_{\rho_1}$, we have
$\partial P, \partial Q \in \mathcal{K}_\gamma$. In the latter case,
$\partial P$ and $\partial Q$ are already in
$\mathcal{K}_\gamma$. In the latter case, $P$ and $Q$ are already in $\mathcal{K}_\gamma$. 

\end{itemize}
\end{proof}
\begin{corollary} \label{corollary:punchline}
Let $A$ and $B$ be orientation preserving isometries of $\mathbb{H}^4$. Then the pair $A$ and $B$ is linked by half-turns if and only if $\mathcal{K}_A \cap \mathcal{K}_B$ is nonempty. \end{corollary} \begin{proof} If $A$ and $B$ are linked by half-turns, then there are planes bounded by $\alpha$, $\beta$ and $\delta$ such that $A=H_\alpha H_\beta$ and $B=H_\beta H_\delta$. But $\beta$ is an element of both $\mathcal{K}_A$ and $\mathcal{K}_B$. \par If $\mathcal{K}_A \cap \mathcal{K}_B$ is nonempty, let $\beta$ be one of its elements. Then there are $\alpha\in \mathcal{K}_A$ and
$\beta\in \mathcal{K}_B$ such that $A=H_\alpha H_\beta$ and
$B=H_\beta H_\delta$.
\end{proof}

\section{Conditions for linking in dimension 4} \label{section:conditions}
\par In this section, some conditions for a pair of isometries in $\mathbb{H}^4$ to be linked are stated and proved. The main idea is to find geometric or computational requirements for a given pair to be linked. The conditions are divided into cases depending on which type of isometry is given. Note that in the definition of linked pairs, the order of the isometries matters. Whereas in this paper, the order of the isometries does not matter once they are linked (See Theorem \ref{theorem:revisedkaran}).

\par The conditions for linking pairs with twisting planes are quite demanding, justifying the theorem of Basmajian and Maskit \cite{arabernard} that linked pairs are of measure zero. Still, pencils can be a useful tool for finding a common perpendicular plane to factorize orientation preserving isometries.
\begin{theorem}
If $A$ and $B$ are pure hyperbolic isometries of $\mathbb{H}^4$ with ultra-parallel axes, then there is a plane $P$ orthogonal to both axes of $A$ and $B$. Hence, $A$ and $B$ are linked.
\end{theorem}
\begin{proof}
\par Suppose both $A$ and $B$ are pure hyperbolic with ultra-parallel axes. Let $L$ be the common perpendicular line between the axes of $A$ and $B$. If the axes of $A$ and $B$ lie in a plane $C$, then $L \subset C$ and there are plenty of planes that are orthogonal to $C$ containing $L$. Pick $P$ to be any of those planes. Then $\partial P \in \mathcal{K}_A\cap \mathcal{K}_B$. If the axes of $A$ and $B$ do not lie in the same plane, the axis of $A$ and $L$ still lie in a plane $P_1$. As $L$ and the axis of $B$ intersect in a single point, so does $P_1$ and axis of $B$ and they are contained is a unique hyperplane $C$. There is a unique plane $P$ orthogonal to $C$ through $L$. Since both axes of $A$ and $B$ lie in the hyperplane $C$, $P$ must be orthogonal to both of them. Then $\partial P \in \mathcal{K}_A\cap \mathcal{K}_B$. It follows that there are $P_1,P_2 \in \mathcal{K}_A$ and $P_3,P_4 \in \mathcal{K}_B$ such that $A= H_{P_1}H_{P}=H_{P}H_{P_2} $ and $B=H_{P_3}H_{P}=H_{P}H_{P_4}$.
\end{proof}

\begin{theorem}
Let $A$ and $B$ be pure parabolic isometries of $\mathbb{H}^4$. Then $A$ and $B$ are linked.
\end{theorem}
\begin{proof}
\par Suppose their fixed points $x_A$ and $x_B$ are not equal. Then there is a unique line $L$ connecting $x_A$ to $x_B$. Let $h_A \in \mathcal{F}_A$ be the element containing $x_B$ and $h_B \in \mathcal{F}_B$ the element containing $x_A$. Then $h_B\cap h_A$ contains the line $L$. Either $h_A\cap h_B$ is a circle or $h_A=h_B$. Let $P$ be a circle in $h_A\cap h_B$ through $x_A$ and $x_B$. Thus $P$ is a circle in different or same spheres in $\mathcal{F}_A$ and $\mathcal{F}_B$. Then $P\in \mathcal{K}_A \cap \mathcal{K}_B$. Hence, there are $P_1,P_2 \in \mathcal{K}_A$ and $P_3,P_4 \in \mathcal{K}_B$ such that $A= H_{P_1}H_{P}=H_{P}H_{P_2} $ and $B=H_{P_3}H_{P}=H_{P}H_{P_4}$.
\par If $x_A=x_B$, there are still $h_A\in \mathcal{F}_A$ and $h_B\in \mathcal{F}_B$ that have three possibilities: $h_A = h_B$; $P=h_A \cap h_B$ is a circle; or $h_A \cap h_B =\{x_A\}$. In the first case, let $t\in \mathcal{T}_A$. Then $h_B \cap t \in \mathcal{K}_A \cap \mathcal{K}_B$ and so $A$ and $B$ are linked. In the second case, $P \in \mathcal{K}_A \cap \mathcal{K}_B $ and links $A$ and $B$. In the last case, let $ \mathcal{F}_A = \mathcal{F}_B$ which implies that $ \mathcal{K}_B= \mathcal{K}_A$ whose elements make $A$ and $B$ linked.
\end{proof}

\begin{theorem}
Let $A$ and $B$ be isometries of $\mathbb{H}^4$. Suppose $A$ is pure hyperbolic, and $B$ is pure parabolic and $\mathrm{fix}(A)\cap \mathrm{fix}(B)=\emptyset$. Then $A$ and $B$ are linked.
\end{theorem}
\begin{proof}
\par Let $v$ be the fixed point of $B$ and $L$ be the axis of $A$.
Then there is an $h_x\in \mathcal{F}_A$ containing $v$.
If $\hat h$ is the hyperplane bounded by $h_x$, there is
a unique point $a \in \hat h\cap L$. Then there is a unique
$h_a \in \mathcal{F}_B$ which bounds a hyperplane that
contains $a$. Since the hyperplanes bounded by $h_x$ and $h_a$
intersect in $a$, their intersection is at least a plane
orthogonal to $L$ and bounded by $v$. Otherwise, $h_x = h_a$
and one can pick a plane $P \subseteq \hat h_x \cap \hat h_a$
bounded by $v$ and passes through $a$. Then $\partial P\in
\mathcal{K}_A \cap \mathcal{K}_B $ so $A$ and $B$ are linked.
\end{proof}

Let $A$ be a pure parabolic isometry of $\mathbb{H}^4$ and $B$ a pure loxodromic isometry of $\mathbb{H}^4$. Suppose the fixed points of $A$ and $B$ in $\widehat{\mathbb{R}^3}$ are disjoint. The following are conditions for $A$ and $B$ to be linked.
\begin{theorem}[Condition 1]
Suppose there is an $h\in \mathcal{F}_A\cap \mathcal{F}_B$. Then $A$ and $B$ are linked.
\end{theorem}
\begin{proof}
Let $L$ be the boundary of the twisting plane of $B$, and let $x$ be the fixed point $A$. Since $L\in \mathcal{D}_B$, $h$ intersects $L$ in two points $y$ and $z$. If $x$ is equal to either $y$ or $z$, then $L$ is the unique element of $\mathcal{D}_B$ containing $x$. It allows any plane or sphere $t$ in $\widehat{\mathbb{R}^3}$ containing $L$ to be an element of both $\mathcal{T}_A$ and $\mathcal{R}_B$. Hence, $h\cap t$ is an element of both $\mathcal{K}_A$ and $\mathcal{K}_B$ that links $A$ and $B$. If $x$, $y$, and $z$ are three distinct points, they form a unique circle $L_2$ that must be a subset of $h$. Then $L_2$ is perpendicular to $L$ so it is in $\mathcal{K}_B$. It also passes through $x$ so $L_2 \in \mathcal{K}_A$. Thus $L_2 \in \mathcal{K}_A \cap \mathcal{K}_B$ so $A$ and $B$ are linked.
\end{proof}

\begin{theorem}[Condition 2]
If the fixed point of $A$ is in the boundary of the twisting plane of $B$, then $A$ and $B$ are linked.
\end{theorem}
\begin{proof}
Let $x$ be the fixed point of $A$. Since $x$ is not the fixed point of $B$, there is a unique $h_x \in \mathcal{F}_B$ containing $x$. Let $L$ be the boundary of the twisting plane of $B$. Then $L \in \mathcal{D}_B$ intersects $h$ in two points, one of which is $x$. Let $m$ be the other intersection point. There is also a unique $h_m \in \mathcal{F}_A$ containing $m$. Either $h_m=h_x$ which is done in the first condition, or $h_m\cap h_x$ is a circle $c$ since both $h_m$ and $h_x$ contain $x$ and $m$. We must show $c\in \mathcal{K}_A \cap \mathcal{K}_B$. The intersection points of $c$ and $L$ are $x$ and $m$, but $c\subset h_x \in \mathcal{F}_B$ so $c\in \mathcal{K}_B$. As $A$ is pure parabolic, any circle inside any element of $\mathcal{F}_A$ and passing through $x$ is an element of $\mathcal{K}_A$. Hence, $c\in \mathcal{K}_A$ as $x \in c \subset h_m \in \mathcal{F}_A$ and $x\in c$. Then other planes can be found so that $A$ and $B$ are linked.
\end{proof}

\begin{theorem}[Condition 3] Let $L \subset \widehat{\mathbb{R}^3}$ be the boundary of the twisting plane of $B$. If three points in $L$ form a subset of $h$ for some $h\in \mathcal{F}_A$, then $A$ and $B$ are linked.
\end{theorem}
\begin{proof}
The three points determine $L$ and are sufficient for $L$ to fit inside $h$. Let $x \in \widehat{\mathbb{R}^3}$ be the fixed point of $A$. The case where $x\in L$ is is handled in the previous condition. We may assume $x\notin L$. Then there is a unique $c_x \in \mathcal{K}_B$ containing $x$. We must show $c_x \in \mathcal{K}_A$. As $c_x \in \mathcal{K}_B$, it must intersect $L$ in two points which are in $h$. But $h$ contains $x$ since $A$ is parabolic, so it has three points in common with $c_x$. They are sufficient for $c_x$ to be a subset of $h$. A circle or line in $h$ containing $x$ is an intersection of $h$ with some element of $\mathcal{T}_A$. Hence $c_x$ having this property, must be in $\mathcal{K}_A$. This implies that $A$ and $B$ are linked.
\end{proof}

Let $A$ be a screw parabolic and $B$ a be pure hyperbolic isometry of $\mathbb{H}^4$. Assume both have disjoint fixed points in the boundary. Throughout this case, let $x$ be the fixed point of $A$ and let $L$ be the boundary of its twisting plane. Likewise, let $d_x$ be the unique element of $\mathcal{D}_B$ containing $x$.
\begin{theorem}[Condition 1] If there is an $h\in \mathcal{F}_A \cap \mathcal{F}_B$, then $A$ and $B$ are linked.
\end{theorem}
\begin{proof}
Let $a$ and $b$ be points in $\widehat{\mathbb{R}^3}$ so that $d_x\cap h=\{x,b\}$ and $L\cap h= \{x,a\}$. If $a=b$, then $d_x=L$ so any sphere $c$ containing $L$ is an element of both $\mathcal{T}_B$ and $\mathcal{R}_A$. It follows that $c\cap h$ is an element of both $\mathcal{K}_A$ and $\mathcal{K}_B$. If $a\neq b$, the points $a$, $b$, and $x$ form a unique circle $L_2$ in $\widehat{\mathbb{R}^3}$. Since $a, b, x \in h$, $L_2\subset h$. We claim that $L_2 \in \mathcal{K}_A \cap \mathcal{K}_B$. First, $L_2$ intersects $L$ in two points, $a$ and $x$, while sitting in $h \in \mathcal{F}_A$. Hence, $L_2\in \mathcal{K}_A$. Next, $L_2$ intersects $d_x$ in two points, $b$ and $x$ while sitting in $h\in \mathcal{F}_B$. Thus, $L_2 \in \mathcal{K}_B$. The half-turn around $L_2$ is the desired common factor of $A$ and $B$ linking them.
\end{proof}

\begin{theorem} [Condition 2] If $L \subset h$ for some $h \in \mathcal{F}_B$, then $A$ and $B$ are linked.
\end{theorem}
\begin{proof}
The hypothesis implies that $h$ is the unique element of $\mathcal{F}_B$ containing $x$. Let $b$ be the intersection of $d_x$ and $h$ in $\widehat{\mathbb{R}^3}$ other than $x$. There is a unique $h_b \in \mathcal{F}_A$ containing $b$. Let $a$ be the intersection of $h_b$ and $L$ other than $x$. Connect $a$, $b$ and $x$ with the unique circle $L_2$. Then $L_2 =h_b \cap h$ while $h\in \mathcal{R}_A$ and $h_b\in \mathcal{T}_B$. It also connects $L$ and $d_x$ in two distinct points. Hence, $L_2 \in \mathcal{K}_A \cap \mathcal{K}_B$ so $A$ and $B$ are linked.
\end{proof}

\begin{theorem} [Condition 3]
If either $m\in L$ or $d_x$ is orthogonal to $k_m$, then $A$ and $B$ are linked.
\end{theorem}
The definitions of $m$ and $k_m$ are as follows. There is a unique element $h_x \in \mathcal{F}_B$ containing $x$, and it is orthogonal to $d_x$ through two intersection points, one of which is $x$. Let $m$ be the other intersection point. If $m\notin L$ there are unique elements $h_m\in \mathcal{F}_A$ and $k_m\in \mathcal{K}_A$ containing $m$.
\begin{proof}
Suppose $m\in L$. It is possible that $d_x=L$ which implies that $h_m=h_x$ reducing the case to the previous condition. If $d_x\neq L$, $h_m\cap h_x$ is a circle $c$ containing $m$ and $x$. It follows that $c\subset \mathcal{K}_A \cap \mathcal{K}_B$ since it intersects $d_x$ and $L$ orthogonally through two points. Hence, $A$ and $B$ are linked.
\par If $m\notin L$, the hypothesis requires that $d_x$ is orthogonal to $k_m$. It follows that $k_m\subset h_x$ since all circles orthogonal to $d_x$ at $x$ and $m$ are contained in $h_x$. Therefore, $k_m \in \mathcal{K}_B$ so $A$ and $B$ are linked.
\end{proof}

\begin{theorem}[Computational Condition] Let $A$ be a screw parabolic isometry of $\mathbb{H}^4$ leaving the $z$-axis of $\widehat{\mathbb{R}^3}$ invariant and fixing $\infty$. Let $B$ be a pure hyperbolic isometry of $\mathbb{H}^4$ fixing $v$ and $w$ in $\mathbb{R}^3$. If $v$ and $w$ are equidistant to the $z$-axis then $A$ and $B$ are linked.
\end{theorem}
If $v=(v_1,v_2,v_3)$ and $w=(w_1,w_2,w_3) \in \mathbb{R}^3$, we say that $v$ and $w$ are equidistant to the $z$-axis if $v_1^2+v_2^2=w_1^2+w_2^2$.
\begin{proof}
The unique element of $d_\infty \in \mathcal{D}_B$ containing
$\infty$ is the Euclidean line connecting $v$ and $w$. Let $m$
be the midpoint of $v$ and $w$. Thus,
\[m=\biggl(\frac{v_1+w_1}{2}, \frac{v_2+w_2}{2},
\frac{v_3+w_3}{2} \biggr).\]
Let $(m_1, m_2, m_3)=m$. Then the horizontal Euclidean line $k_m$
connecting $m$ and the $z$-axis through $(0,0,m_3)$ is an element
of $\mathcal{K}_A$. The angle between $k_m$ and $d_\infty$ can be
computed from the inner product $(-m_1, -m_2, 0)\cdot(v-m)$, which
simplifies to $-\frac 14 \bigl(v_1^2-w_1^2\bigr)-\frac 14
\bigl( v_2^2 -w_2^2\bigr)$. The hypothesis implies that this is
zero so $k_m$ and $d_\infty$ are perpendicular and hence $k_m \in
\mathcal{K}_B$. Then $A$ and $B$ are linked.
\end{proof}

Let $A$ be screw parabolic and $B$ be a pure parabolic isometries of $\mathbb{H}^4$ with different fixed points. Throughout this case, let $x$ be the fixed point of $A$ and $y$ be the fixed point of $B$. Suppose $L$ is the boundary of the twisting plane of $A$. There are unique elements $d_x\in \mathcal{D}_B$ containing $x$ and $d_y\in \mathcal{D}_A$ containing $y$. If $y\notin L$, denote the unique element of $\mathcal{K}_A$ containing $y$ by $k_y$.

\begin{theorem} [Condition 1] If there is an $h\in \mathcal{F}_A \cap \mathcal{F}_B$, then $A$ and $B$ are linked.
\end{theorem}
\begin{proof}
As $A$ and $B$ are both parabolic, $h$ contains both $x$ and $y$. Since $L\in \mathcal{D}_A$, it intersects $h$ orthogonally in a point $a$ other than $x$. If $a=y$, then $L=d_x$ so any circle in $h$ connecting $x$ and $a$ is an element of $\mathcal{K}_A \cap \mathcal{K}_B$ making $A$ and $B$ linked. If $a\neq y$ the points $a$, $x$ and $y$ form a circle $c$ in $h$. The circle $c$ intersects $L$ in two points and also $d_x$ in $x$ and $y$. Hence, $c$ is in both $\mathcal{K}_A $ and $ \mathcal{K}_B$ so $A$ and $B$ are linked.
\end{proof}

\begin{theorem} [Condition 2] If $y\in L$, then $A$ and $B$ are linked.
\end{theorem}
\begin{proof}
There are spheres $h_x\in \mathcal{F}_B$ and $h_y\in\mathcal{F}_A$ containing both $x$ and $y$. If $h_x=h_y$, $A$ and $B$ are linked as per previous condition. Otherwise, $h_x\cap h_y$ is a circle $k$ passing through $x$ and $y$. If $y\in L$, $k\in\mathcal{K}_A$ but $k$ is also in $\mathcal{K}_B$ so $A$ and $B$ are linked.
\end{proof}

\begin{theorem} [Condition 3] If $L\subset h$ for some $h\in \mathcal{F}_B$, then $A$ and $B$ are linked.
\end{theorem}
\begin{proof}
We may assume $y\notin L$. Otherwise, the linking of $A$ and $B$ is implied by the previous condition. Then there is a unique $k_y\in \mathcal{K}_A$ containing $y$. It intersects $L$ in a point $p$ other than $x$. The circle $k_y$ is uniquely determined by the points $x,p,y \in h$. Hence, $k_y\subset h$ so $k_y\in \mathcal{K}_B$ since $B$ is pure parabolic. The half-turn about $k_y$ links $A$ and $B$.
\end{proof}

\begin{theorem} [Condition 4] If $k_y\perp d_x$, then $A$ and $B$ are linked.
\end{theorem}
If $y\notin L$, there is a unique $k_y \in \mathcal{K}_A$ containing $y$. If $y\in L$, there are many choices for $k_y\in \mathcal{K}_A$ but choosing $k_y\perp d_x$ is not necessary.
\begin{proof}
There is a unique $h_x\in \mathcal{F}_B$ containing $x$. It intersects $d_x$ in $x$ and $y$ which are also in $k_y$. Every circle orthogonal to $d_x$ through $x$ and $y$ must be a subset of $h_x$ so $k_y\subset h_x$. The span of $k_y\cup d_x$ is an element of $\mathcal{T}_B$ so $k_y\in \mathcal{K}_B$. Recall that $k_y \in \mathcal{K}_A$. Therefore, $A$ and $B$ are linked.
\end{proof}

Linking pairs of hyperbolic isometries are better expressed inside $\mathbb{H}^4$. The hyperplanes and planes in $\mathbb{H}^4$ bounded by the pencils can intersect the invariant planes or lines of an isometry.
\par Let $A$ be a pure loxodromic isometry of $\mathbb{H}^4$ with axis $\mathrm{Ax}_A$ and twisting plane $\tau_A$. Let $B$ be a pure hyperbolic isometry of $\mathbb{H}^4$ with axis $\mathrm{Ax}_B$. Assume $\tau_A$ and $\tau_B$ are disjoint. Then there is a unique line $N$ perpendicular to both $\mathrm{Ax}_A$ and $\mathrm{Ax}_B$. It intersects $A$ in a point $a$ and $B$ in a point $b$. There is a unique hyperplane $h_a$ orthogonal to $\mathrm{Ax}_A$ through $a$. Likewise, there is a unique hyperplane $h_b$ orthogonal to $\mathrm{Ax}_B$ through $b$. Since $\mathrm{Ax}_A \subset \tau_A$, there is a unique line $L_a\subset \tau_A$ perpendicular to $\mathrm{Ax}_A$ through $a$. Throughout this case, $A$, $B$, $\mathrm{Ax}_A$, $\mathrm{Ax}_B$, $\tau_A$, $N$, $a$, $b$ and $L_a$ are used.
\begin{theorem} [Condition 1] If $L_a=N$, then $A$ and $B$ are linked.
\end{theorem}
\begin{proof}
The hypothesis implies that $\tau_A$ intersects $\mathrm{Ax}_B$ at least at $b$. If $\mathrm{Ax}_B \subset \tau_A$, there are plenty of planes orthogonal to $\tau_A$ through $N$. Any of them is orthogonal to $\mathrm{Ax}_B$. If $\mathrm{Ax}_B$ intersects $\tau_A$ only at $b$, they span a hyperplane $h$. There is a plane $k$ orthogonal to $h$ through $N$. Since $h$ contains $\tau_A$ and $\mathrm{Ax}_B$, $k$ is also orthogonal to $\tau_a$ and $\mathrm{Ax}_B$. (By dimension count, $k\subset h_b$.) Hence, $k$ is an element of both $\mathcal{K}_A$ and $\mathcal{K}_B$, so $A$ and $B$ are linked.
\end{proof}

\begin{theorem} [Condition 2] If $L_a\subset h_b$, then $A$ and $B$ are linked.
\end{theorem}
\begin{proof}
If $L_a\subset h_b$, it is possible that $L_a=N$, which is the previous case. We may assume that $L_a\neq N$. Then $L_a$ and $N$ span a plane $k$ which sits in $h_a$. It follows that $k\subset h_b$ since both $L_a$ and $N$ lie in $h_b$. Thus $k$ is orthogonal to $\mathrm{Ax}_B$ and $\tau_A$ so $k\in \mathcal{K}_B$ linking $A$ and $B$.
\end{proof}

Let $A$ and $B$ be screw parabolic isometries of $\mathbb{H}^4$, with disjoint fixed points $x$ and $y$ respectively. Let $\tau_A$ and $\tau_B$ be their respective twisting planes. Define $L_A =\partial \tau_A$ and $ L_B=\partial\tau_B$. As $x\neq y$, there are unique elements $h_x\in \mathcal{F}_B$ and $h_y\in\mathcal{F}_A$ such that $x\in h_x$ and $y\in h_y$. The conditions in which a common orthogonal plane exists are quite restrictive, so $A$ and $B$ are highly unlikely to be linked.
\begin{theorem} [Condition 1]
If there is $h\in \mathcal{F}_A \cap \mathcal{F}_B$ and the points $x$, $y$, $x_h$ and $y_h$ form a circle, then $A$ and $B$ are linked.
\end{theorem}
Any $h\in \mathcal{F}_A$ intersect $L_A$ in two points. Let $x_h$ be the intersection point other than $x$. If $h\in \mathcal{F}_B$, let $y_h$ be the element of $h\cap L_B$ other than $y$.
\begin{proof}
Let $c$ be the circle formed by $x$, $y$, $x_h$ and $y_h$ as allowed by the hypothesis. We must show $c \in \mathcal{K}_A \cap \mathcal{K}_B$. Since $\{x,y,x_h,y_h\} \subset h$, $c\subset h$. Also $c$ is orthogonal to $L_A$ through $\{x,x_h\}$ so $c\in \mathcal{K}_A$. Similarly, $c$ is orthogonal to $L_B$ through $\{y,y_h\}$ so $c \in \mathcal{K}_B$. Hence $A$ and $B$ are linked.
\end{proof}

\begin{theorem}[Condition 2] If $y\in L_A$ and $x \in L_B$, then $A$ and $B$ are linked.
\end{theorem}
\begin{proof}
We may assume $h_x\neq h_y$; otherwise $L_A=L_B$ making any circle in $h_x$ that connects $x$ to $y$ an element of $\mathcal{K}_A \cap \mathcal{K}_B$. Let $c$ be $h_x\cap h_y$ which is a circle containing $\{x,y\}$. The hypothesis implies that $h_y\cap L_A =\{x,y\}=h_x\cap L_B$, so $c$ is orthogonal to both $L_A$ and $L_B$. Therefore, $c \in \mathcal{K}_A \cap \mathcal{K}_B$. The existence of $c$ makes $A$ and $B$ linked.
\end{proof}

\begin{theorem} [Condition 3] If $h_y\cap L_A\subset h_x$, $h_x\cap L_B \subset h_y$ and $h_x\neq h_y$, then $A$ and $B$ are linked.
\end{theorem}
\begin{proof}
Let $c$ be $h_x\cap h_y$. Then $c$ is a circle containing $\left(h_y\cap L_A\right)\cup \left(h_x\cap L_B\right)$. It is orthogonal to $L_B$ and $L_A$ through two distinct points each. Hence, $ c \in \mathcal{K}_A \cap \mathcal{K}_B$. It follows that $A$ and $B$ are linked.
\end{proof}

\begin{theorem} [Condition 4] If $L_B \subset h_y$ and $L_A \subset h_x$, then $A$ and $B$ are linked.
\end{theorem}
\begin{proof} Since $L_B \subset h_y$, $h_y \in \mathcal{R}_B$. Let $c=h_y\cap h_x$. Then $c\in \mathcal{K}_B$. Similarly, $L_A\subset h_x$ implies that $h_x\in \mathcal{R}_A$ so $c \in \mathcal{K}_A$. Thus $A$ and $B$ are linked.
\end{proof}

Let $A$ and $B$ be pure loxodromic isometries of $\mathbb{H}^4$ with axes $\mathrm{Ax}_A$, $\mathrm{Ax}_B$ and twisting planes $\tau_A$, $\tau_B$ respectively. Assume $\mathrm{Ax}_A$ and $\mathrm{Ax}_B$ are disjoint. There is a unique line $N$ perpendicular to both $\mathrm{Ax}_A$ and $\mathrm{Ax}_B$. There are hyperplanes $h_a$ orthogonal to $\mathrm{Ax}_A$ through $a$, and $h_b$ orthogonal to $\mathrm{Ax}_B$ through $b$. There are also lines $L_a\subset \tau_A$ perpendicular to $\mathrm{Ax}_A$ through $a$, and $L_b\subset \tau_B$ perpendicular to $\mathrm{Ax}_B$ through $b$. Like the previous case, $A$ and $B$ are rarely linked.
\begin{theorem} If $L_a$ and $L_b$ are coplanar, then $A$ and $B$ are linked.
\end{theorem}
Since $\tau_A$ and $\tau_B$ are disjoint, the lines $N$, $L_a$,
and $L_b$ are distinct. The plane $P$ containing $L_a$ and $L_b$
is unique. Since $N$ connects $a$ and $b$, $N$ lies in $P$. The
axis of $A$ is perpendicular to both $N$ and $L_a$ so $\tau_A$ is
orthogonal to $P$ through $N$. Likewise, $\tau_B$ is orthogonal
to $P$ through $N$. Then $\partial P \in \mathcal{K}_A \cap \mathcal{K}_B$.
\begin{proof}

\end{proof}

Let $A$ be screw parabolic isometry of $\mathbb{H}^4$ fixing $x$ and with twisting plane $\tau_A$. Let $B$ be a pure loxodromic isometry of $\mathbb{H}^4$ with axis $\mathrm{Ax}_B$ and twisting plane $\tau_B$. Suppose $x$ does not bound $\mathrm{Ax}_B$. There is a unique $h_x \in \mathcal{F}_B$ that contains $x$. There is also $d_x\in \mathcal{D}_B$ containing $x$. Let $x_2$ be the intersection of $d_x$ with $h_x$ other than $x$. Since $x\neq x_2$, there are unique elements $d_2\in \mathcal{D}_A$ and $h_2 \in \mathcal{F}_A$ containing $x_2$.
\par Let $L_A$ be the boundary of $\tau_A$ and $L_B$ be that of $\tau_B$. Since $L_A \in \mathcal{D}_A$, $L_A\cap h_2$ has exactly two points. Let $a$ be the element of $L_A\cap h_2$ other than $x$. Similarly, let $b_1$ and $b_2$ be the elements of $L_B\cap h_x$. Throughout this case, $A$, $B$, $\tau_A$, $\tau_B$, $\mathrm{Ax}_B$, $x$, $h_x$, $x_2$, $d_x$, $d_2$, $h_2$, $L_A$, $L_B$, $b_1$ and $b_2$ are used consistently.
\begin{theorem} [Condition 1]
If there is an $h \in \mathcal{F}_A \cap \mathcal{F}_B$ and the points of $h\cap\left(L_A \cup L_B\right)$ form a circle, then $A$ and $B$ are linked.
\end{theorem}
\begin{proof}
Since $x\in h$, $h$ is the unique element of $\mathcal{F}_B$ containing $x$. That is $h=h_x$. It follows that $x_2\in h$ and $h$ is the unique element of $\mathcal{F}_A$ containing $x_2$. So $h=h_x=h_2$. Hence, \begin{align*}
h\cap\left(L_A \cup L_B\right) &= \left(h\cap L_A\right)\cup \left(h\cap L_B\right) \\
\nonumber &= \left(h_2\cap L_A\right)\cup \left(h_x\cap L_B\right) \\
\nonumber &= \left\{ x,a,b_1,b_2 \right\}.
\end{align*}
Let $c$ be the circle containing $\left\{ x,a,b_1,b_2 \right\} $ according to the hypothesis. Then $c\subset h$ since $\left\{ x,a,b_1,b_2 \right\} \subset h$. It is orthogonal to $L_A$ through $\{x,a\}$ and to $L_B$ through $\{b_1,b_2\}$. Therefore, $c \in \mathcal{K}_A \cap \mathcal{K}_B$. It implies that $A$ and $B$ are linked.
\end{proof}

\begin{theorem} [Condition 2] If $a\in h_x$; $b_1,b_2\in h_2$; and $h_x\neq h_2$, then $A$ and $B$ are linked.
\end{theorem}
\begin{proof}
Let $c=h_x\cap h_2$. Then $c$ is a circle since $h_x\neq h_2$. It is orthogonal to $L_A$ through $\{a,x\}$ and to $L_B$ through $\{b_1,b_2\}$. Thus, $c \in \mathcal{K}_A \cap \mathcal{K}_B$ so $A$ and $B$ are linked.
\end{proof}

\begin{theorem} [Condition 3]
If there are $p_A\in \mathcal{F}_A$ and $p_B\in \mathcal{F}_B$ such that $L_B\subset p_A$ and $L_A\subset p_B$, then $A$ and $B$ are linked.
\end{theorem}
\begin{proof}
Since $L_A\subset p_B$, we have $p_B\in \mathcal{R}_A$, so $p_B\neq p_A$. Let $c=p_B\cap p_A$. Then $c\in \mathcal{K}_A$ as it is an intersection of a pair in $\mathcal{F}_A\times \mathcal{R}_A$. Likewise, $L_B\subset p_A$ implies that $p_A\in \mathcal{R}_B$, so $c$ is also in $\mathcal{K}_B$. Hence $A$ and $B$ are linked.
\end{proof}

\section{Pairs without invariant subplane} \label{section:disproveara}

\par It is possible for a pair to be linked by half-turns without having invariant lower dimensional planes. This section shows how to construct a linked pair that does not have invariant lower dimensional space. It can be started by investigating individual isometries then proving that the only invariant planes of an isometry are those that are bounded by the elements of its invariant pencil (if such exist), twisting pencil (if the rotational part is an involution), axis, twisting plane or twisting hyperplane  (the hyperplane orthogonal to the twisting plane through the axis).
\par One way to find all invariant subplanes is to express an isometry of $\mathbb{H}^n$ as a matrix in $\mathrm{PSO}(n,1)$. The eigenvectors correspond to fixed points if they are not space-like in the Lorentzian space
$\mathbb{R}^{n,1}$. The following are examples of hyperbolic isometries.
\begin{align*}
\begin{bmatrix}
1&0&0&0&0 \\
0&1&0&0&0 \\
0&0&1&0&0 \\
0&0&0&\sqrt{2}&1 \\
0&0&0&1&\sqrt{2}
\end{bmatrix}
 \text{ and } &
\begin{bmatrix}
\cos \theta & -\sin \theta &0&0&0\\
\sin \theta & \cos \theta &0&0&0\\
0&0&1&0&0 \\
0&0&0&\sqrt{2}&1 \\
0&0&0&1&\sqrt{2}
\end{bmatrix}
\end{align*}
The upper left
$3\times 3$ entries can be replaced by any nontrivial element of
$\mathrm{SO}(3)$ and the resulting matrix is pure loxodromic. The construction shows that the invariant lower dimensional subspaces are exactly the axis, twisting plane, twisting hyperplane, and the elements of the invariant and dual pencils. Moreover, the permuted and dual pencils can also be located.
\par For elliptic isometries, the upper left
$4\times 4$ entries of the identity matrix
$\mathrm{Id}_5$ can be replaced by an element of
$\mathrm{O}(4)$ and the resulting matrix is an elliptic element. Analyzing the invariant subspaces is a matter of studying $\mathrm{O}(4)$ and $\mathrm{SO}(4)$.


\subsection{Parabolic isometries of $\mathrm{PSO}(n,1)$}
A parabolic element of $\mathrm{PSO}(2,1)$ that fixes the light-like vector
$(0,1,1)$ has the form
\[
P_2= \begin{bmatrix}
1&-t&t\\
t&1-\frac{t^2}{2}&t^2/2\\
t&-t^2/2&1+\frac{t^2}{2}
\end{bmatrix}  \text{ where } t \neq 0.
 \]
It can be decomposed into Jordan canonical form
$P_2=S_2J_2S_2^{-1}$ where
\[
S_2=
\begin{bmatrix}
0&t&-t/2\\
t^2&0&-1\\
t^2&0&0
\end{bmatrix}
 \text{ and }  J_2 =
\begin{bmatrix}
1&1&0\\
0&1&1\\
0&0&1
\end{bmatrix}.
\]
The form of
$P_2$ can be extended to
$\mathrm{PSO}(3,1)$ using
$x,y \in \mathbb{R}$ with
$x^2+y^2>0$. Specifically, let
\begin{align*}
P_3&=\begin{bmatrix}
1&0&-x&x\\
0&1&-y&y\\
x&y&1-\frac{x^2+y^2}{2}&\frac{x^2+y^2}{2}\\
x&y&-\frac{x^2+y^2}{2}&1+\frac{x^2+y^2}{2}
\end{bmatrix}.
\end{align*}
Then
$P_3$ fixes
$(0,0,1,1)$ and has Jordan decomposition
$P_3=S_3J_3S_3^{-1}$ where
\[S_3=
\begin{bmatrix}
-\frac{xy}{x^2+y^2}&0&x&0\\
-\frac{x^2}{x^2+y^2}&0&y&0\\
-y/2&x^2+y^2&\frac{x^2+y^2}{2}&0\\
-y/2&x^2+y^2&\frac{x^2+y^2}{2}&1\\
\end{bmatrix} \text{ and }
 J_3=
\begin{bmatrix}
1&0&0&0\\
0&1&1&0\\
0&0&1&1\\
0&0&0&1
\end{bmatrix}.
\]
Notice that
$P_3$ has a space-like eigenvector whereas
$P_2$ has only light-like eigenvector. If
$v_1$ and
$v_2$ are the second and third column vectors of
$S_3$, then
$P_3$ maps any vector
$v \in \mathbb{R}^{3,1}$ into a linear combination of
$v$, $v_1$ and $v_2$. In particular if
$v$ is time-like, then
$P_3$ leaves
$\mathrm{span}\{v,v_1,v_2\}$ invariant. Hence the $3$-dimensional time-like vector subspaces that contain
$v_1$ and $v_2$ project to the hyperbolic planes that form the invariant pencil of
$P_3$.
\par The extension of
$P_3$ to
$\mathrm{PSO}(4,1)$ is
$P_4$ defined given by
\begin{align*}
P_4&=\begin{bmatrix}
1&0&0&-x&x\\
0&1&0&-y&y\\
0&0&1&-z&z\\
x&y&z&1-\frac{x^2+y^2+z^2}{2}&\frac{x^2+y^2+z^2}{2}\\
x&y&z&-\frac{x^2+y^2+z^2}{2}&1+\frac{x^2+y^2+z^2}{2}
\end{bmatrix}.
\end{align*}
The Jordan decomposition of $P_4$ is $S_4J_4S^{-1}_4$ where $S_4=$
\[
\begin{bmatrix}
 0& -\frac{xy}{x^2+y^2+z^2} & 0 &x &0 \\
 -z/y & \frac{x^2+z^2}{x^2+y^2+z^2} & 0 &y &0 \\
 1& -\frac{yz}{x^2+y^2+z^2} & 0 &z &0 \\
0&-y/2 & x^2+y^2+z^2 & \frac{x^2+y^2+z^2}{2} &0 \\
0&-y/2 & x^2+y^2+z^2 & \frac{x^2+y^2+z^2}{2} &1
\end{bmatrix} \text{ and } \] \[
 J_4=
\begin{bmatrix}
1&0 &0 &0 &0 \\
0&1 &0 &0 &0 \\
0&0 & 1& 1&0 \\
0&0 &0 &1 &1 \\
0&0 & 0& 0& 1
\end{bmatrix}.
\]

Then
$P_4$ has two space-like eigenvectors which correspond to the first two columns of
$S_4$. Let
$v_1$ and $v_2$ be the third and fourth columns of
$S_4$. Like
$P_3$, any vector
$v \in \mathbb{R}^{4,1}$ is mapped by
$P_4$ into a linear combination of
$v$, $v_1$ and $v_2$. Thus, the invariant pencil of $P_4$  is precisely the set of intersections of the hyperboloid
$\mathbb{H}^4$ with the
$4$-dimensional time-like vector subspaces containing
$\{v_1,v_2\}$. The dual pencil of
$P_4$ can be located by the
$3$-dimensional vector time-like subspaces containing $\{v_1,v_2\}$.
\par Let
$v_3$ be any time-like vector in
$\mathbb{R}^{4,1}$ and
$\rho$ be an element of $\mathrm{SO}(5)$ that fixes $\mathrm{span}\{v_1,v_2,v_3\}$ pointwise. The Lorentz orthogonal complement of $\mathrm{span}\{v_1,v_2,v_3\}$ is two dimensional space-like that is rotated by
$\rho$. Then
$\rho\in\mathrm{PSO}(4,1)$ and
$\rho P_4$ is a screw parabolic isometry of $\mathrm{PSO}(4,1)$. If the rotation angle of $\rho$ is not an integer multiple of $\pi$, then only the scalar multiples of
$(0,0,0,1,1)$ are the eigenvectors of
$\rho P_4$. Hence the only time-like subspace left invariant by
$\rho P_4$ is the span of
$\{v_1,v_2,v_3\}$. The other invariant subspaces of
$\rho P_4$ are either light-like or space-like.
\subsection{Mismatched invariant planes} \label{subsection:lorentzpencils}
\par If two isometries have no common invariant line, plane or hyperplane, they can still be linked. The elements of the half-turn bank of an isometry are not invariant circles/planes, so a pair of isometries can still have a common element in their half-turn banks.
\begin{theorem} \label{theorem:disprovearasorry}
There exists half-turn linked pair of isometries of $\mathbb{H}^4$ that leaves neither a point, a line, a plane nor a hyperplane invariant.
\end{theorem}
\begin{proof}
Start with
$\widehat{\mathbb{R}^3}$ which is a model for the boundary of upper half-space
$\mathbb{H}^4$. Let $A$ be a loxodromic isometry that fixes $(0,0,1)$ and $(0,0,-1)$ but rotates the vertical line connecting them by $\pi/6$ angle. Let $B$ be the Poincar\'e extension of the function sending $x$ to $x+(1,0,1)$. Note that $B$ does not have to be pure parabolic; it can be a hyperbolic isometry fixing $(1,0,1)$ and $(-1,0,-1)$.
\par Then the horizontal line connecting $(0,0,0)$ to $(0,1,0)$ is an element of $\mathcal{K}_A \cap \mathcal{K}_B$. The invariant planes of $B$ do not include the twisting plane of $A$. Also the unit sphere, which bounds the twisting hyperplane of $A$, is not left invariant by $B$ since it does not contain the fixed point of $B$. There is no invariant hyperbolic line either as the fixed points of $A$ and $B$ do not match.
\end{proof}

\section{Acknowledgement} The author thanks Ara Basmajian and Karan Puri who shared significant amount of their time providing productive insights and also his
thesis adviser, Jane Gilman.
Some of the results here are from the author's Ph.D.~thesis. The author thanks the Rutgers-Newark Mathematics Department for its support during the course of his graduate studies and the Graduate School Newark for a Dissertation Fellowship.

\end{document}